 \documentclass[10pt,twocolumn,twoside]{IEEEtran}

\usepackage{cite}
\usepackage{amsmath,amssymb,amsfonts}
\usepackage{algorithmic}
\usepackage{hyperref}
\usepackage{textcomp}
\usepackage{graphicx,color}
\usepackage{mathrsfs}
\usepackage{subfigure}
\usepackage{url}
\usepackage{color}
\usepackage{dsfont}
\usepackage{bbm}
\usepackage{booktabs}
\usepackage{array}
\usepackage[table]{xcolor}
\usepackage{yfonts}

\newtheorem{theorem}{Theorem}[section]
\newtheorem{lemma}[theorem]{Lemma}

\newtheorem{remark}{Remark}

\newtheorem{example}{Example}

\newcommand{\setdef}[2]{\{#1 \; : \; #2\}}
\newcommand{\subscr}[2]{{#1}_{\textup{#2}}}

\newcommand{\until}[1]{\{1,\dots,#1\}}

\newcommand{\blkdiag}{\mathrm{blk}\text{-}\mathrm{diag}}

\newcommand*\dif{\mathop{}\!\mathrm{d}} 

\newcommand{\real}{\mathbb{R}}

\newcommand{\transpose}{\mathsf{T}} 

\newcommand{\mc}{\mathcal}

\newcommand{\1}{\mathds{1} }

\DeclareSymbolFont{bbold}{U}{bbold}{m}{n}
\DeclareSymbolFontAlphabet{\mathbbold}{bbold}

\newcommand\oprocendsymbol{\hbox{$\square$}}
\newcommand\oprocend{\relax\ifmmode\else\unskip\hfill\fi\oprocendsymbol}

\newcommand*{\QEDA}{\hfill\ensuremath{\blacksquare}}%
\newenvironment{pfof}[1]{\vspace{1ex}\noindent{\itshape Proof of
    #1:}\hspace{0.5em}} {\hfill\QEDA\vspace{1ex}}
\newcommand{\QED}{\hfill \mbox{\raggedright \rule{.1in}{.1in}}}
\newenvironment{proof}{\vspace{1ex}\noindent{\itshape Proof:}\hspace{0.5em}}
{\hfill\QED\vspace{1ex}}

\begin{document}
\title{Stability Conditions for Cluster Synchronization in Networks of Heterogeneous Kuramoto Oscillators} \author{Tommaso Menara, Giacomo Baggio, Danielle S. Bassett, and Fabio Pasqualetti 
\thanks{This material is based upon
    work supported in part by ARO award 71603NSYIP, and in part by NSF
    awards BCS1430279 and BCS1631112.} \thanks{ Tommaso Menara,
    Giacomo Baggio and Fabio Pasqualetti are with the Department of
    Mechanical Engineering, University of California at Riverside,
    \{\href{mailto:tomenara@engr.ucr.edu}{\texttt{tomenara}},
    \href{mailto:gbaggio.ucr.edu}{\texttt{gbaggio}},
    \href{mailto:fabiopas@engr.ucr.edu}{\texttt{fabiopas\}@engr.ucr.edu.}} Danielle
    S. Bassett is with the Department of Bioengineering, the
    Department of Electrical and Systems Engineering, the Department of Physics and Astronomy,
    the Department of Psychiatry, and the Department of Neurology, University of
    Pennsylvania,
    \href{mailto:mailto:dsb@seas.upenn.edu}{\texttt{dsb@seas.upenn.edu.}}}}

\maketitle

\begin{abstract}
  In this paper we study cluster synchronization in networks of
  oscillators with heterogenous Kuramoto dynamics, where
  multiple groups of oscillators with identical phases coexist in a
  connected network. Cluster synchronization is at the basis of
  several biological and technological processes; yet the
    underlying mechanisms to enable cluster synchronization of
    Kuramoto oscillators have remained elusive. In this paper we
  derive quantitative conditions on the network weights, cluster
  configuration, and oscillators' natural frequency that ensure
  asymptotic stability of the cluster synchronization manifold; that
  is, the ability to recover the desired cluster synchronization
  configuration following a perturbation of the oscillators'
  states. Qualitatively, our results show that cluster synchronization
  is stable when the intra-cluster coupling is sufficiently stronger
  than the inter-cluster coupling, the natural frequencies of the
  oscillators in distinct clusters are sufficiently different, or, in
  the case of two clusters, when the intra-cluster dynamics is
  homogeneous. We illustrate and validate the effectiveness of our
  theoretical results via numerical studies.
  \end{abstract}


\section{Introduction}
\IEEEPARstart{S}{ynchronization} refers broadly to patterns of
coordinated activity that arise spontaneously or by design in several
natural and man-made systems
\cite{MG-MEJN:02,AP-MR-JK:03,Lewis2014}. Examples include coherent
firing of neuronal populations in the brain \cite{JC-EH-OS-GD:11},
coordinated flashing of fireflies \cite{Moiseff2000}, flocking of
birds \cite{Giardina2008}, exchange of signals in wireless networks
\cite{OS-US-YBN-SS:08}, consensus in multi-agent systems
\cite{FG-LS:10}, and power generation in the smart grid
\cite{FD-MC-FB:13}. Synchronization enables complex functions: while
some systems require complete (or full) synchronization among all the
components in order to function properly, others rely on cluster (or
partial) synchronization, where different groups exhibit different,
yet synchronized, internal behaviors~\cite{FS-EO:07}.

While studies of full synchronization are numerous and have generated
a rich literature, e.g., see \cite{SB-VL-YM-MC-DUH:06, NC-MWS:08,
  AA-ADG-JK-YM-CZ:08}, conditions explaining the onset of cluster
synchronization and its properties are less well
  understood. Such conditions are necessary for the analysis and,
more importantly, the control of synchronized activity
across biological \cite{KK:94,
    LS-RO-BB-AH-BC:02, IB-MH:11} and technological  \cite{JRT-KST-DJD-GDV-SZ-PA-RR:99} systems. For instance, a deeper
understanding of the mechanisms enabling cluster synchronization might
not only shed light on the nature of the healthy human brain
\cite{AS-JG:05}, but also enable and guide targeted interventions for
patients with neurological disorders, such as epilepsy
\cite{Lehnertz2009} and Parkinson's disease \cite{Hammond2007}.

We study cluster synchronization in networks of
  oscillators with Kuramoto dynamics \cite{Kuramoto1975}, which,
  despite their apparent simplicity, are particularly suited to
  represent complex synchronization phenomena in neural systems
  \cite{AD-VWB:11}, as well as in many other natural and technological
  systems \cite{FD-MC-FB:13}. Although our study and modeling choices
are guided by the practical need to understand and control patterns of
synchronized functional activity in the human brain, as they naturally
arise in healthy and diseased populations
\cite{GD-VK-ARM-OS-RK:09,FV-MS-PJH-GS-JC-RL:15}, in this paper we
focus on developing the mathematical foundations of a quantitative
approach to the analysis and control of cluster synchronization in a
weighted network of  Kuramoto oscillators.
In particular, we derive conditions on the oscillators'
  coupling and their natural frequencies that guarantee the stability
  of an arbitrary cluster configuration.

\noindent\textbf{Related work} Cluster
synchronization, where multiple synchronized groups of oscillators
coexist in a connected network, is an exciting phenomenon that has
attracted the attention of the physics, dynamical systems, and
controls communities, among others. Existing work on this topic has
shown that cluster-synchronized states can be linked to the existence
of certain network symmetries \cite{VNB-IB-MH:00, AYP-HS-HN:02,
  IB-VNB-KN-MH:03, IS-MG-MP:03, AYP:08} or symmetries in the nodes'
dynamics \cite{DF-GR-MDB:17}.  More recently, in \cite{Pecora2014,
  Sorrentino2016}, the stability of cluster states corresponding to
network symmetries is addressed with the Master Stability Function
approach \cite{MLP-TLC:98}. In contrast to this previous work,
\cite{GR-JJES:11} combines network symmetries with contraction
analysis to study the stability of synchronized states. Further
studies relating contraction properties and cluster synchronization
are conducted in \cite{QCP-JJS:07,ZA-BD-END-NEL:17}. Finally, control
algorithms for cluster synchronization are developed in
\cite{WW-WZ-TC:09, WL-BL-TC:10}. To the best of our knowledge,
however, the above studies are not applicable to oscillators with
Kuramoto dynamics, which we study in this work.

A few papers have studied cluster synchronization of Kuramoto
oscillators. Specifically, in \cite{CF-AC-FP:17,YQ-YK-MC:18} the
authors provide invariance conditions for an approximate definition of
cluster synchronization and for particular types of
networks. Invariance of exact cluster synchronization, which is the
notion used in this paper, is also studied in
\cite{Schaub2016,LT-CF-MI-DSB-FP:17}. Stability of exact cluster
synchronization is investigated in \cite{IVB-BNB-VNB:16} where,
however, only the restrictive case of two clusters for identical
Kuramoto oscillators with inertia is considered, and in
\cite{YSC-TN-AEM:17}, where only implicit and numerical stability
conditions are provided. To the best of our
knowledge, our work presents the first explicit analytical conditions
for the (local) stability of the cluster synchronization manifold in
sparse and weighted networks of heterogeneous Kuramoto oscillators.

\noindent\textbf{Paper contribution} The main contribution of this
paper is to characterize conditions for the stability of cluster
synchronization in networks of oscillators with Kuramoto dynamics. We
consider a notion of exact cluster synchronization, where the phases
of the oscillators within each cluster remain equal to each other over
time, and different from the phases of the oscillators in the other
clusters. We derive three conditions. First, we show that the cluster
synchronization manifold is locally exponentially stable when the
intra-cluster coupling is sufficiently stronger than the inter-cluster
coupling. We quantify this tradeoff using the theory of perturbation
for dynamical systems together with the invariance properties of
cluster synchronization. Second, through a Lyapunov argument, we show
that the cluster synchronization manifold is locally exponentially
stable when the natural frequencies of the oscillators in disjoint
clusters are sufficiently different (in their limit to
infinity). Third, we focus on the case of two clusters, and provide a
quantitative condition on the network weights and oscillators' natural
frequency for the stability of the cluster synchronization
manifold. This analysis shows that asymptotic stability of the cluster
synchronization manifold is guaranteed for weak inter-cluster weights,
sufficiently different natural frequencies, or even homogeneous
intra-cluster configurations.

As minor contributions, we provide examples showing that
  network symmetries are not necessary for cluster synchronization of
  Kuramoto oscillators, and a sufficient condition guaranteeing the
  absence of stable synchronization submanifolds.

\noindent\textbf{Paper organization} The rest of the paper is
organized as follows. Section \ref{sec: setup} contains our problem
setup and some preliminary notions. Section \ref{sec: section 3}
contains our main results; that is, our conditions for the stability
of the cluster synchronization manifold in Kuramoto networks. Finally,
section \ref{sec: conclusion} concludes the paper, and the Appendix
contains the proofs of our results.

\noindent\textbf{Mathematical notation} The set $\real_{>0}$ (resp. $\real_{<0}$) denotes the positive
(resp. negative) real numbers, whereas the sets $\mathbb{S}^1$ and
$\mathbb{T}^n$ denote the unit circle and the $n$-dimensional torus,
respectively. The vector of all ones is represented by $\1$. We let
$O(f)$ denote the order of the function $f$. Further, we denote a
positive (resp. negative) definite matrix $A$ with $A\succ 0$ (resp.
$A\prec 0$). We indicate the smallest (resp. largest) eigenvalue of a
symmetric matrix with $\lambda_\text{min}(\cdot)$ (resp. 
$\lambda_\text{max}(\cdot)$). A \mbox{(block-)diagonal} matrix
is represented by ($\mathrm{blk}\text{-})\mathrm{diag}(\cdot)$. We let $\| \cdot \|$
denote the $\ell^2$-norm, and $\mathsf{i} = \sqrt{-1}$.  Finally, $A^\dagger$ represents the
Moore-Penrose pseudoinverse of the matrix $A$.

\section{Problem setup and preliminary notions}\label{sec: setup} In
this work we characterize the stability properties of certain
synchronized trajectories arising in networks of oscillators with
Kuramoto dynamics. To this aim, let $\mc G = (\mc V, \mc E)$ be the
connected and weighted graph representing the network of oscillators,
where $\mc V = \until{n}$ and $\mc E \subseteq \mc V \times \mc V$
represent the oscillators, or nodes, and their interconnection edges,
respectively. Let $A = [a_{ij}]$ be the weighted adjacency matrix of
$\mc G$, where $a_{ij} \in \real_{> 0}$ is the weight of the edge
$(i , j) \in \mc E$, and $a_{ij} = 0$ when $(i, j) \not\in \mc E$. The
dynamics of $i$-th oscillator~is
\begin{align}\label{eq: kuramoto}
  \dot \theta_i = \omega_i + \sum_{j \neq i} a_{ij} \sin(\theta_{j}-\theta_{i}),
\end{align}
where $\omega_i \in \real_{> 0}$ and $\theta_i \in \mathbb{S}^1$ denote
the natural frequency and the phase of the $i$-th oscillator. Unless specified differently, we assume that the edge weights are symmetric. That is,
\begin{itemize}
\item[(A1)] The network adjacency matrix satisfies $A = A^\transpose$.
\end{itemize}
Assumption (A1) is typical in the study of (cluster) synchronization
in networks of Kuramoto oscillators, e.g., see \cite{AJ-NM-MB:04,
  FD-FB:12, Doerfler2014}, as it facilitates the derivation of
stability results.  While relaxing this assumption is beyond the scope
of this work, we will discuss how our stability results can also be
applied to study cluster synchronization with asymmetric network
weights (see Remark \ref{remark: asymmetric}). Finally, since the
diagonal entries of the adjacency matrix $A$ do not contribute to the
dynamics in \eqref{eq: kuramoto}, we assume that $\mc G$ does not
contain self-loops.

A network exhibits cluster synchronization when the oscillators can be
partitioned so that the phases of the oscillators in each cluster
evolve identically. To be precise, let
$\mc P = \{\mc P_1, \dots, \mc P_m\}$, with $m >1$, be a partition of
$\mc V$, where $\bigcup_{i=1}^m \mc P_i = \mc V$ and
$\mc P_i \cap \mc P_j = \emptyset$ if $i \neq j$. Define the
\emph{cluster synchronization
  manifold} associated with the partition $\mc P$ as
\begin{align*}
  \mc S_{\mc P} = \setdef{\theta \in \mathbb{T}^n }{ \theta_i =
  \theta_j  \text{ for all } i,j \in \mc P_k, \,k = 1,\dots,m} .
\end{align*}
Then, the network is cluster-synchronized with partition $\mc P$ when
the phases of the oscillators belong to $\mc S_{\mc P}$ at all times.

In this paper we characterize conditions on the network weights and
the oscillators' natural frequency that guarantee \emph{local
  exponential stability} of the cluster synchronization manifold
$\mc S_{\mc P}$, for a given partition $\mc P$.\footnote{Loosely
  speaking, the manifold $\mc S_{\mc P}$ is locally exponentially
  stable if $\theta$ converge to $\mc S_{\mc P}$ exponentially fast
  when $\theta (0)$ is sufficiently close to~$\mc S_{\mc P}$.}  In
order to study stability of the cluster synchronization manifold, we
assume $\mc S_{\mc P}$ to be invariant \cite[Chapter
3]{ANM-LH-DR:08}.\footnote{The manifold $\mc S_{\mc P}$ is invariant
  if $\theta(0) \in \mc S_{\mc P}$ implies $\theta \in \mc S_{\mc P}$
  at all times.} In particular, following \cite{LT-CF-MI-DSB-FP:17},
invariance of $\mc S_{\mc P}$ is guaranteed by the following
conditions:
\begin{itemize}
\item[(A2)] Given $\mc P = \{\mc P_1, \dots , \mc P_m\}$, the natural
  frequencies satisfy $\omega_i = \omega_j$ for every
  $i,j \in \mc P_k$ and $k~\in~\until{m}$,\footnote{This
      condition is necessary for $\mc S_{\mc P}$ to be forward
      invariant, and thus stable \cite{LT-CF-MI-DSB-FP:17}, and is
      motivated by observed synchronization phenomena, e.g., see
      \cite{DM-MGP-CDG-GLR-MC:07}.} and
  
\item[(A3)]  The network weights satisfy
  $\sum_{k \in \mc P_\ell} a_{i k} - a_{j k} = 0$ for every
  $i,j \in \mc P_z$ and $z,\ell \in \until{m}$, with $z \neq \ell$.
\end{itemize}
Thus, in the remainder of the paper we assume that (A2) and (A3) are
satisfied for the network partition being considered.

  \begin{figure}[t]
  \centering 
  \subfigure[]
  {
    \includegraphics[width=0.405\columnwidth]{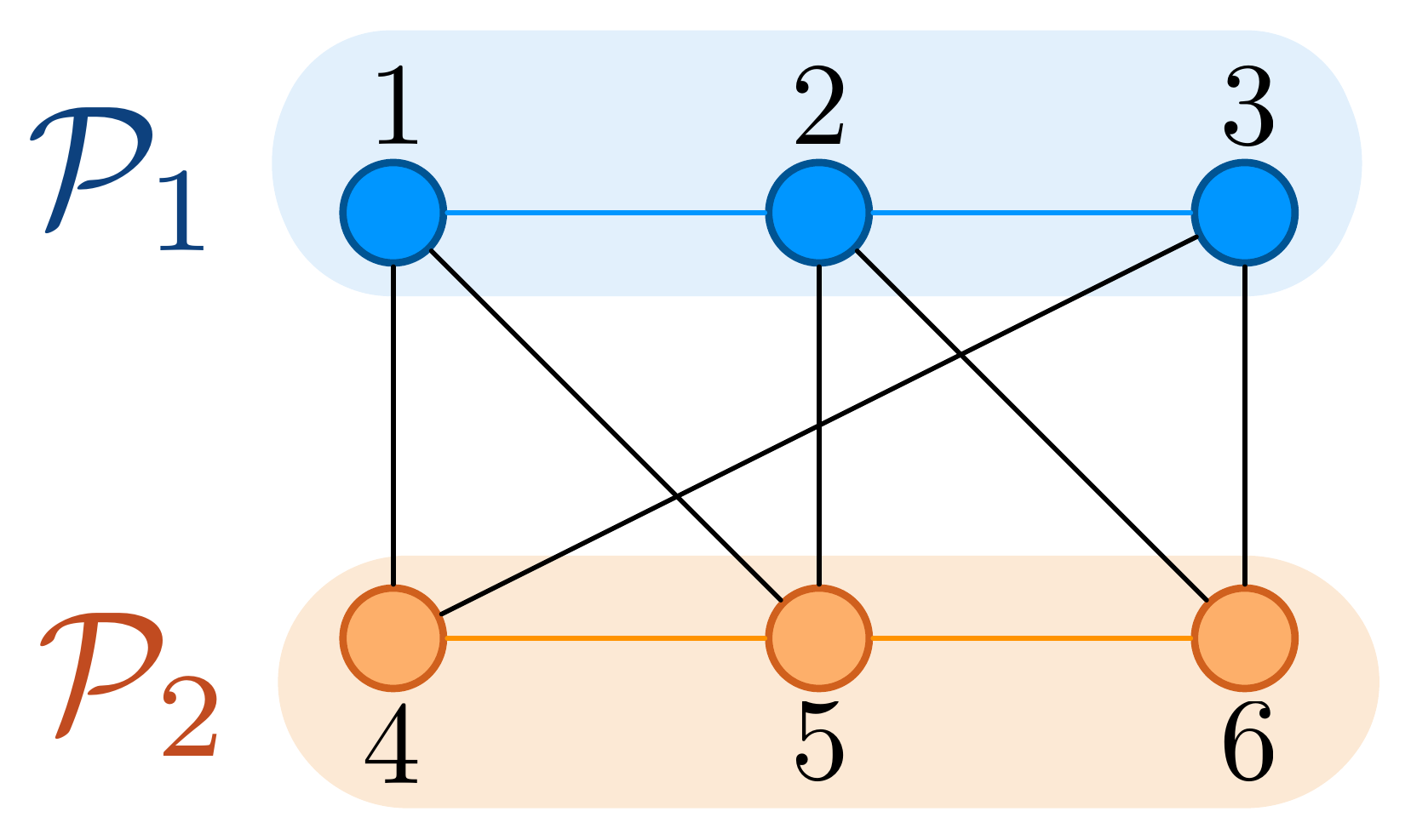}
  \label{fig: aa}
  }
    \;
    \subfigure[]{
      \includegraphics[width=.48\columnwidth]{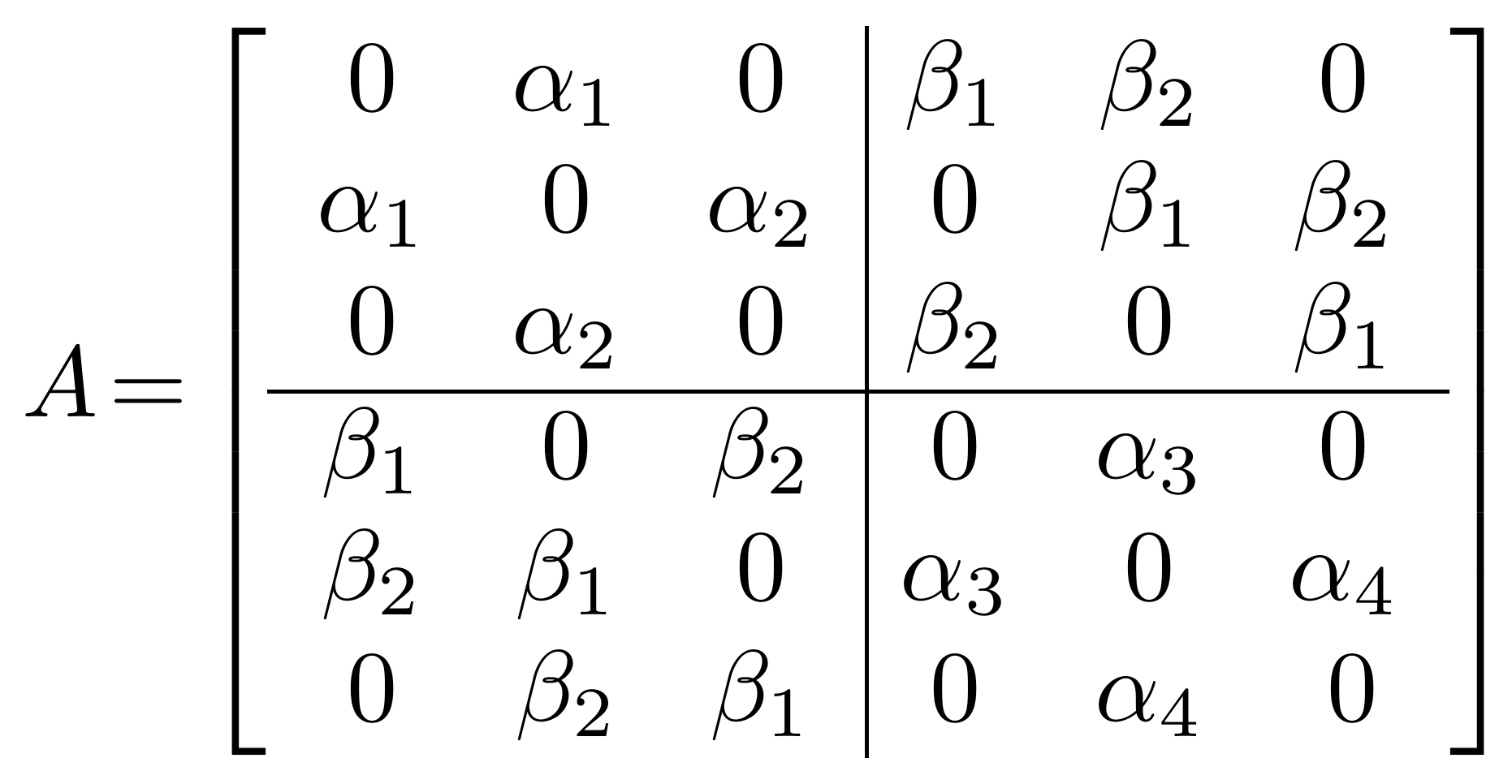}
  \label{fig: bb}
  }
  \caption{Fig. \ref{fig: aa} illustrates a network of
      $6$ oscillators with adjacency matrix as in Fig. \ref{fig:
        bb}. In this network, the partition
      $\mc P=\{\mc P_1, \mc P_2\}$, which satisfies Assumption (A3),
      cannot be identified by group symmetries of the network for any
      choice of the positive weights $\alpha_1$, $\alpha_2$,
      $\alpha_3$, $\alpha_4$, $\beta_1$ and $\beta_2$. The manifold $\mc S_{\mc P}$ is invariant whenever the oscillators' natural frequencies satisfy Assumption (A2). Thus, this example shows that network symmetries are not necessary for
      cluster synchronization of Kuramoto oscillators.
      }
\label{fig: group symmetry}
\end{figure}

\begin{figure}[t]
  \centering 
  \subfigure[]
  {
     \includegraphics[width=.401\columnwidth]{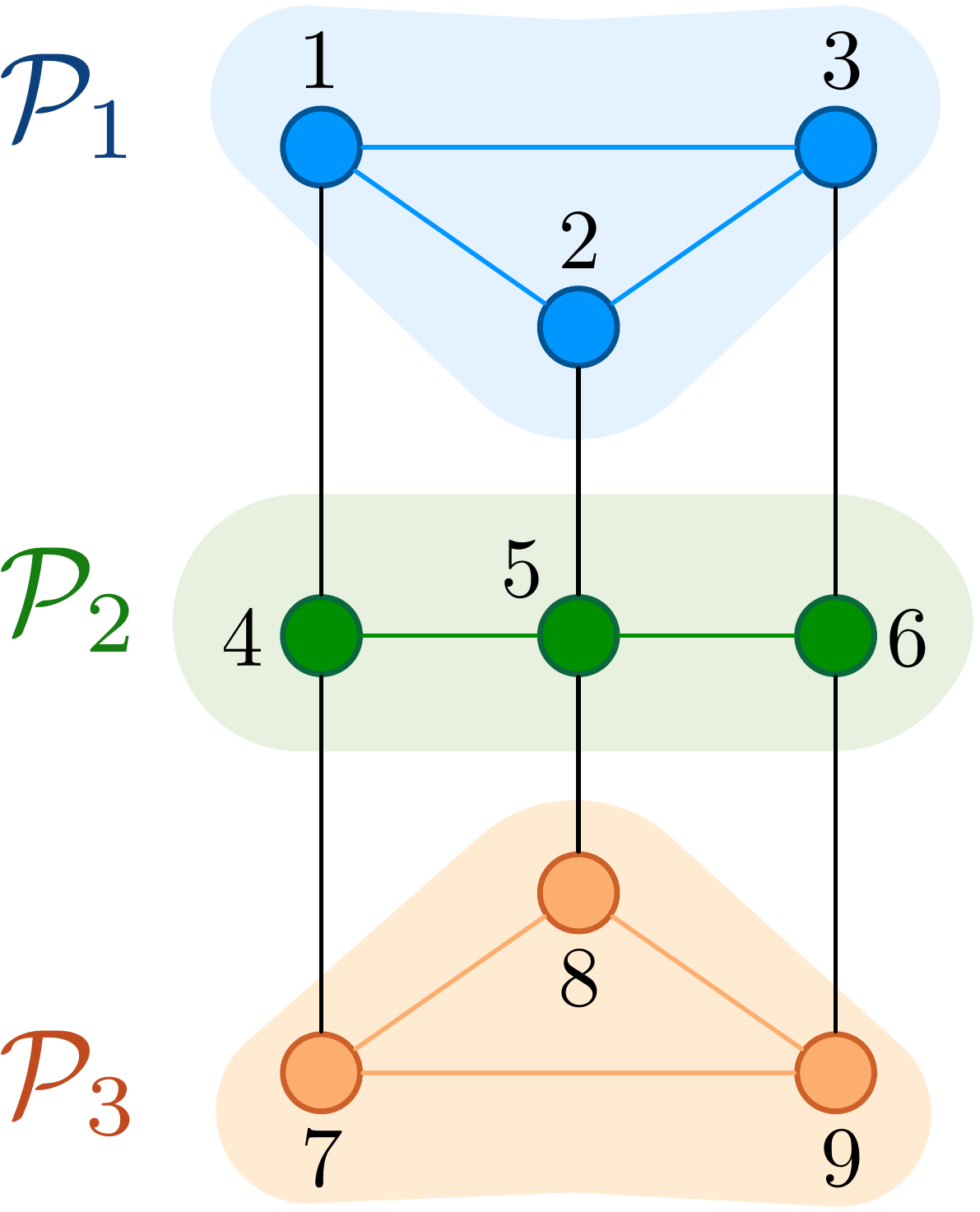}
  \label{fig: quotient_graph}
  }\;
    \subfigure[]{
       \includegraphics[width=.49\columnwidth]{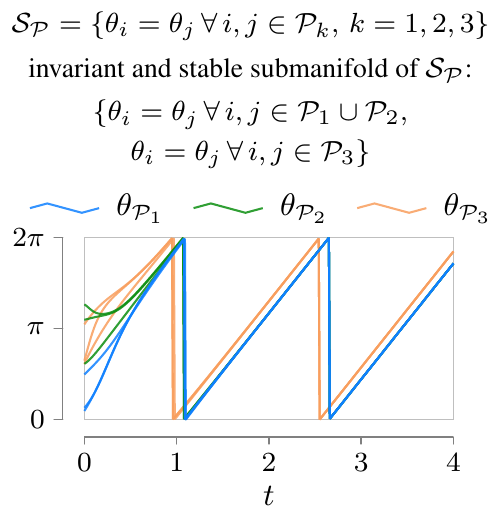}
  \label{fig: evolution submanifold}
  }
  \caption{Fig. \ref{fig: quotient_graph} illustrates a
      network with partition $\mc P = \{ \mc P_1, \mc P_2, \mc
      P_3\}$. As shown in Fig. \ref{fig: evolution submanifold}, the
      phases of the oscillators in $\mc P_1$ and $\mc P_2$ have the
      same value over time, showing that a submanifold of
      $\mc S_{\mc P}$ is invariant and stable. For this simulation, we
      use $\omega_1 = 4$, $\omega_2 = 2$, $\omega_3 = 6$,
      $a_{14} = 3$, and $a_{47} = 5$.}  \label{fig:
    quotient_graph_example}
\end{figure}

\begin{remark}{\bf \emph{(Network symmetries, equitable partitions,
      and balanced weights)}}\label{remark: balanced vs symmetries}
  Conditions to ensure the invariance of the cluster synchronization
  manifold have been linked to network symmetries, which are defined
  by the group comprising all node permutations that leave the network
  topology unchanged, e.g., see \cite{Pecora2014, Sorrentino2016,
    YSC-TN-AEM:17}. In Fig.~\ref{fig: group symmetry} we propose a
  network with two clusters, which are not defined by any group
  symmetry, that satisfies Assumption (A3) and thus admits an
  invariant cluster synchronization manifold. This example shows that
  cluster synchronization of Kuramoto oscillators does not require
  symmetric networks. Our Assumption (A3), and in fact the equivalent
  notion of external equitable partition \cite{Schaub2016}, is less
  restrictive than requiring partitions satisfying group symmetries
  \cite{ZM-ZL-GZ:06, Belykh2008, ABS-LP-JDH-FS:18}. Finally,
  Assumptions (A2) and (A3) are necessary when the natural frequencies
  in distinct clusters are sufficiently different (see
  \cite{LT-CF-MI-DSB-FP:17} and Remark \ref{remark: submanifolds}).
  \oprocend
\end{remark}

  \begin{remark}{\bf \emph{(Invariance of submanifolds of
        $\mc S_{\mc P}$)}}\label{remark: submanifolds}
    When the network of oscillators is cluster-synchronized
    (i.e. $\theta(t) \in \mc S_{\mc P}$ for all $t\ge 0$),
    submanifolds of $\mc S_{\mc P}$ may appear whenever the phases
    belonging to two (or more) disjoint clusters have equal
    values (see Fig.~\ref{fig: quotient_graph_example}). Interestingly,
    the example in Fig.~\ref{fig: quotient_graph_example} also points
    out that Assumption (A3) may not be necessary for the invariance
    of $\mc S_{\mc P}$ if the clusters do not evolve with different
    frequencies (see Assumption (A1) in \cite{LT-CF-MI-DSB-FP:17}). In
    what follows we show that, if the natural frequencies of the
    oscillators in disjoint clusters are sufficiently different,
    invariant, and hence stable, submanifolds cannot exist. To see
    this, assume that the phases of the disjoint clusters $\mc P_\ell$
    and $\mc P_z$ remain equal over time. Then, using Assumption (A2)
    and (A3), the dynamics
    \begin{align}\label{eq: difference dynamics remark}
      \dot \theta_\ell - \dot \theta_z =\, \omega_\ell - \omega_z 
      + \sum_{k=1}^m &\left[ \left(\sum_{r \in \mc P_k} a_{\ell
                       r}\right) \sin(\theta_k - \theta_\ell)\right. \notag\\ 
                     &\!\!\!-\left.\left(\sum_{r \in \mc P_k} a_{z
                       r}\right) \sin(\theta_k - \theta_z)\right]\!,
    \end{align}
    must be identically zero, where $\theta_i$ 
    denotes the phase of any oscillator in $\mc P_i$. Clearly, if the following inequality holds,
    \begin{align}\label{eq: condition onset} |\omega_\ell
      -\omega_z| > 2(m-2) \max_{k\ne \ell,z} \left\{\sum_{r \in \mc P_k} a_{\ell r} \,,\sum_{r \in \mc P_k} a_{z r}\! \right\},
    \end{align}
    Equation \eqref{eq: difference dynamics remark} cannot vanish and,
    consequently, the clusters $\mc P_\ell$ and $\mc P_z$ cannot
    evolve with the same phases when the network is cluster
    synchronized.\footnote{In  \eqref{eq: condition onset}, we have $(m-2)$ because
      for $k=z,\ell$, the sine terms in \eqref{eq: difference dynamics
        remark} vanish.} More generally, if condition \eqref{eq:
      condition onset} is satisfied for all pairs of clusters, then
    invariant, and hence stable, cluster synchronization submanifolds
    cannot exist. \oprocend
\end{remark}


We conclude with an example showing that the synchronization manifold
$\mc S_{\mc P}$ is, in general, not globally asymptotically stable due to
the existence of multiple invariant sets.
\begin{example}{\bf \emph{(Multiple invariant
      sets)}}\label{ex: toy example}
  Consider a Kuramoto network with $2N$ oscillators ($N \ge 2$) and with an
  adjacency matrix defined as follows\footnote{This analysis extends
    directly to arbitrary weights $a_{ij} = a$, $a \in\real_{>0}$.}
  (see Fig. \ref{fig: 8n2c} for the case $N = 5$):
  
\begin{figure}[t]
  \centering
    \hspace{0.3cm}\subfigure[]
    {
      \centering
        \includegraphics[width=0.4\columnwidth]{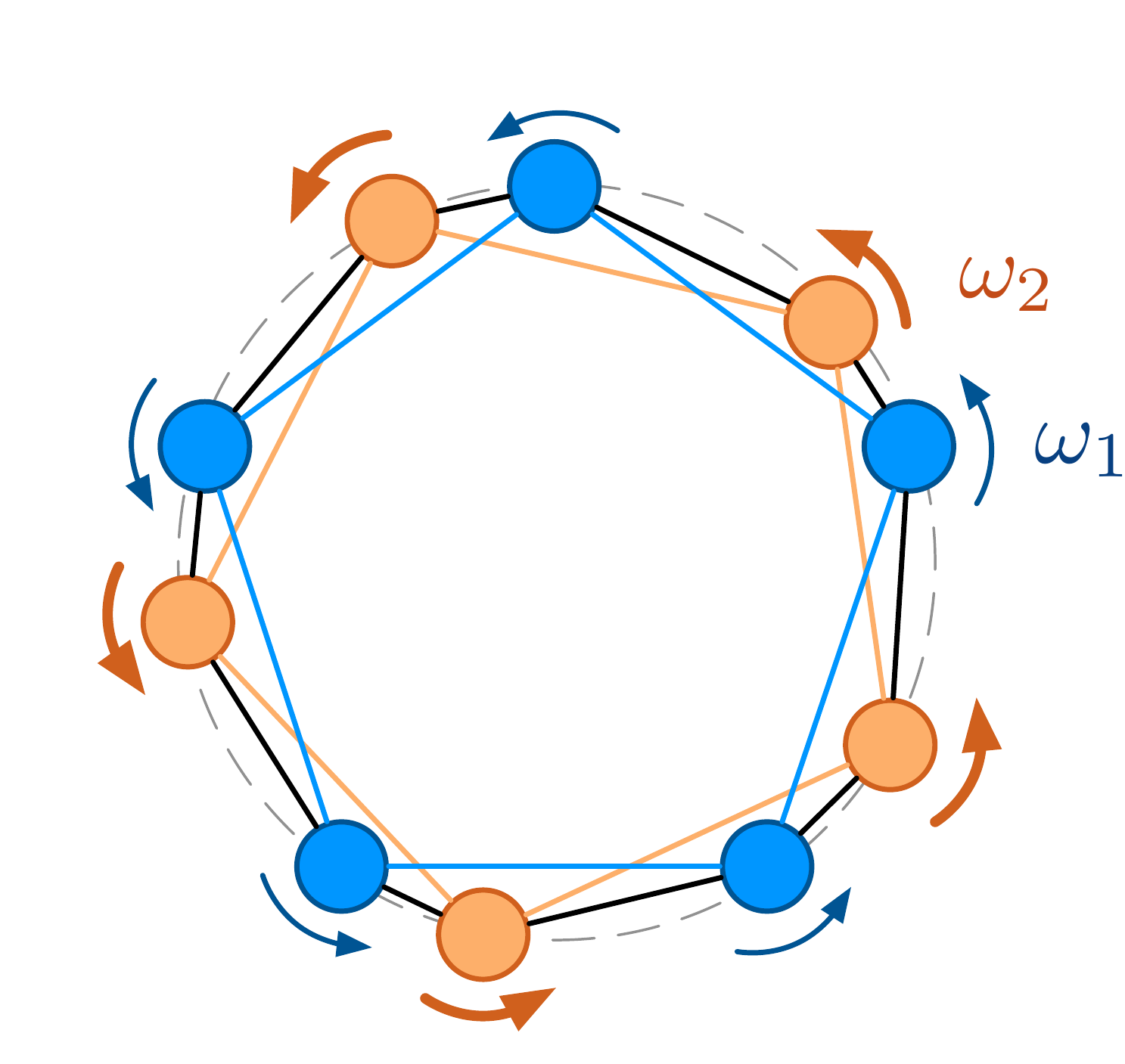}\label{fig: 8n2c}
}\hspace{0.1cm}
    \subfigure[]
    {\centering
    	\includegraphics[width=0.45\columnwidth]{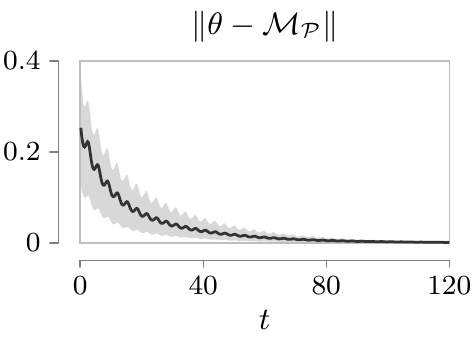}\label{fig: stable Mp}
    }
    \subfigure[]
    {\includegraphics[width=0.45\columnwidth]{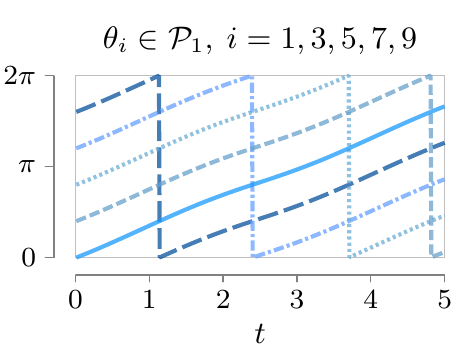}\label{fig: invariant Mp1}
    }
        \subfigure[]
    {\includegraphics[width=0.45\columnwidth]{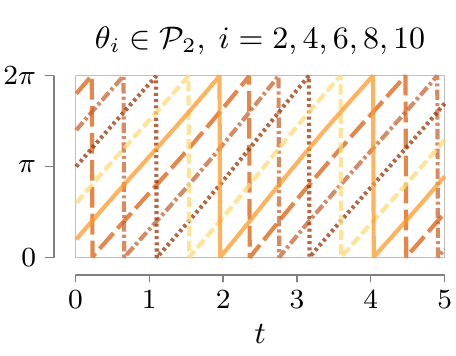}\label{fig: invariant Mp2}
    }
    \caption{Fig. \ref{fig: 8n2c} shows the network in Example \ref{ex: toy example} for the case $N = 5$. The nodes belonging to partition $\mc P_1$ are blue and have natural frequency $\omega_1 = 1$, while the nodes belonging to partition $\mc P_2$ are orange and have $\omega_2 = 3$. 
    Fig. \ref{fig: stable Mp}
    illustrates the stability of the set $\mc M_{\mc P}$ via numerical simulations. We performed $10^3$ iterations, each one with $\theta(0)$ chosen
      randomly within an angle of $\pm 0.01 ~[\text{rad}]$ from
      $\mc M_{\mc P}$. The thick line represents the mean value among all simulations of the $2$-norm distance between $\theta$ and
      $\mc M_{\mc P}$, while the faded area represents the smallest and largest value of the $2$-norm distance between $\theta$ and
      $\mc M_{\mc P}$.
      Fig. \ref{fig: invariant Mp1}-\ref{fig: invariant Mp2} illustrate the invariance of the set $\mc M_{\mc P}$ as the phases in the clusters $\mc P_1$ and $\mc P_2$ evolve respectively with the same~frequencies. 
}
\end{figure}
  \begin{align*}
    a_{ij} =
    \begin{cases}
      1, & \text{ if } |i-j| \le 2,\\
      0, & \text{ otherwise},
    \end{cases}
  \end{align*}
  with $i,j \in \until{2N}$ (and the convention
  $2N + \ell \triangleq \ell$, $-\ell \triangleq 2N + \ell - 1$, for
  $\ell \in \{1,2\}$). Let $\mc P = \{\mc P_1, \mc P_2\}$, with
  $\mc P_1 = \{1,3,\dots,2N-1\}$, $\mc P_2 = \{ 2,4,\dots,2N \}$, and
  define
  \begin{align*}
    \mc M_{\mc P} = \setdef{\theta \in \mathbb{T}^{2N}\! }{\! \theta_{i+2} =
    \theta_i + 2\pi/N, i = 1,\dots, 2N-2}.
  \end{align*}
  It can be verified that Assumption (A3) is satisfied, and that the
  set $\mc S_{\mc P}$ is
  invariant whenever the natural frequencies satisfy Assumption
  (A2). Yet, the set $\mc S_{\mc P}$ is not the only invariant set. In
  fact, $\mc M_{\mc P}$ is also invariant (we prove this by
  showing that $\dot \theta_i = \dot \theta_{i+2}$ when
  $\theta_i,\theta_{i+2} \in \mc M_{\mc P}$):
  \begin{align*}
    \dot\theta_{i} &= \omega_i +  \sin(\theta_{i-2} - \theta_i)
                     +  \sin(\theta_{i+2} - \theta_i)\\
                   &\;\;\;\;+ \sin(\theta_{i-1} - \theta_i)
                     +  \sin(\theta_{i+1} - \theta_i)\\
                   &= \omega_i +  \sin(\theta_i - \theta_{i+2} ) +
                      \sin(\theta_{i+4} - \theta_{i+2}) \\&\;\;\;\;+   \sin(\theta_{i+1} - \theta_{i+2})
                                                                     +
                                                                     \sin(\theta_{i+3}
                                                                     -
                                                                     \theta_{i+2})
                   = \dot \theta_{i+2},
  \end{align*}
  where we have used the fact that
  $\theta_{i+2} - \theta_{i} = 2\pi/N$, and $\omega_i =\omega_j$ for
  all $i,j$ in the same cluster. Further, it can be
  verified numerically that, depending on the number of oscillators
  $N$, the set $\mc M_{\mc P}$ is also locally stable (see
  Fig. \ref{fig: stable Mp}). We conclude that the cluster
  synchronization manifold $\mc S_{\mc P}$ is not, in general,
  globally asymptotically stable. In what follows we derive conditions
  guaranteeing local stability of $\mc S_{\mc P}$. \oprocend
\end{example}

\section{Conditions for the stability of the cluster synchronization
  manifold}\label{sec: section 3}
In this section we derive sufficient conditions for the local
exponential stability of the cluster synchronization manifold. Define
the phase difference $x_{ij} = \theta_j - \theta_i$, and notice that
\begin{align}\label{eq: differences dynamics}
  \dot x_{ij} =  \omega_j-\omega_i+\sum_{z=1}^n \left[a_{jz}\sin(x_{jz})-a_{iz}\sin(x_{iz})\right].
\end{align}
Given a partition $\mc P = \{\mc P_1, \dots ,\mc P_m\}$
of the set $\mc V$ in the graph $\mc G$, we define the following graphs
(see also Example \ref{ex: graphical example}):
\begin{enumerate}
\item the graph of the $k$-th cluster, with $k \in \until{m}$,
  $\mc G_{k} = (\mc P_k, \mc E_{k})$, where
  $\mc E_{k} = \setdef{(i,j)}{(i,j) \in \mc E,\;i,j \in \mc P_k}$;

\item a spanning tree $\mc T_{k} = (\mc P_k, \mc E_{\text{span},k})$
  of $\mc G_{k}$;\footnote{We assume that $\mc G$ and its subgraphs
    $\mc G_k$ are connected. This guarantees the existence of the
    (connected) spanning trees defined in (ii) and (iii). A graph is
    connected if there exists a path between any pair of nodes
    \cite{Godsil2001}.}

\item a spanning tree $\mc T = (\mc V, \mc E_{\mc T})$ of $\mc G$ with
  $\mc E_{\mc T} = \bigcup_{k=1}^m \mc E_{\text{span},k} \cup \mc
  E_\text{inter}$, where $\mc E_\text{inter}$ satisfies
  $| \mc E_\text{inter} | = m-1$.
\end{enumerate}
Further, we define the following vectors of phase differences:
\begin{enumerate}
\setcounter{enumi}{3}
\item $\subscr{x}{intra}^{(k)} = [x_{ij}]$, for all
  $(i,j) \in \mc E_{\text{span},k}$ with $i< j$,

\item
  $\subscr{x}{intra} =
  \left[\subscr{x}{intra}^{(1)\transpose},\dots,\subscr{x}{intra}^{(m)\transpose}\right]^{\transpose}$, and

\item $\subscr{x}{inter} = [x_{ij}]$, for all
  $(i,j) \in \mc E_\text{inter}$ with $i< j$.
\end{enumerate}

It should be noticed that the vectors $\subscr{x}{intra}^{(k)}$,
$\subscr{x}{intra}$ and $\subscr{x}{inter}$ contain, respectively,
$n_{\text{intra},k} = |\mc P_{k}|-1$, $\subscr{n}{intra} = n - m$ and
$\subscr{n}{inter} = m - 1$ entries. Notice that every phase
difference can be computed as a linear function of $\subscr{x}{intra}$
and $\subscr{x}{inter}$. To see this, let $i,j \in \mc V$, and let
$\mathrm{p}(i,j) = \{p_1, \dots, p_\ell\}$ be the unique path on $\mc T$
from $i$ to $j$. Define
$\mathrm{dif}\mathrm{f}(\mathrm{p}(i,j)) = \sum_{k =1}^{\ell-1} s_k$, where
$s_k = x_{p_{k} p_{k+1}}$ if $p_{k}< p_{k+1}$, and
$s_k = - x_{p_{k+1} p_{k}}$ otherwise. Then,
$x_{ij} = \mathrm{dif}\mathrm{f}(\mathrm{p}(i,j))$, and the vectors
$\subscr{x}{intra}$ and $\subscr{x}{inter}$ contain a smallest set of
phase differences that can be used to quantify synchronization among
all of the oscillators in the network. 

Let $B=[b_{k\ell}] \in \real^{|\mc V| \times |\mc E|}$ denote the oriented
incidence matrix of the graph $\mc G = (\mc V, \mc E)$, where $\ell$
corresponds to the edge $(i,j) \in \mc E$, $b_{k\ell} = 1$ if node $k$
is the sink of the edge $\ell$, $b_{k\ell} = -1$ if $k$ is the source
of $\ell$, and $b_{k\ell} = 0$ otherwise.\footnote{Node $i$ is the
  source (resp. sink) of $(i,j)$ if $i < j$ (resp. $i > j$).} Further,
let $B_k$ and $B_{\text{span},k}$ denote the incidence matrices of
$\mc G_{k}$ and $\mc T_{k}$, respectively. Notice that
  $B_{\text{span},k}$ is full rank because it is the incidence matrix
  of an acyclic graph (tree) \cite[Theorem
  8.3.1]{Godsil2001}. Let
$T_{\text{intra},k} = B_{k}^\transpose
(B_{\text{span},k}^{\transpose})^{\dagger}$ be the unique matrix that
maps the phase differences contained in $x_\text{intra}^{(k)}$ to all
intra-cluster phase differences in the $k$-th cluster. That is,
\begin{align}\label{eq: Tintra}
  x^{(k)} = T_{\text{intra},k} \subscr{x}{intra}^{(k)} ,
\end{align}
where
$x^{(k)}$ contains all phase differences in the cluster $\mc P_k$.

We conclude this part by rewriting the intra-cluster dynamics in a
form that will be useful to prove our results. In particular, from the
above discussion and for an intra-cluster phase difference $x_{ij}$ of $x_\text{intra}^{(k)}$, we
rewrite \eqref{eq: differences dynamics} as
\begin{align}\label{eq: Gij}
  \dot x_{ij} =& \underbrace{\sum_{z \in \mc P_k} \left[a_{jz} \sin(\mathrm{dif}\mathrm{f}(\mathrm{p}(j,z))) - a_{iz} \sin(\mathrm{dif}\mathrm{f}(\mathrm{p}(i,z)))\right]}_{F_{ij}^{(k)}(x_\text{intra}^{(k)})} \notag \\
               &+ \underbrace{\sum_{z \not\in \mc P_k} \left[ a_{jz} \sin(\mathrm{dif}\mathrm{f}(\mathrm{p}(j,z))) - a_{iz} \sin(\mathrm{dif}\mathrm{f}(\mathrm{p}(i,z)))\right]}_{G_{ij}^{(k)}(x_\text{intra},x_\text{inter})},
\end{align}
which leads to
\begin{align}\label{eq: intra dynamics k}
  \dot x_{\text{intra}}^{(k)} = F^{(k)}(x_\text{intra}^{(k)}) + G^{(k)}(x_\text{intra},x_\text{inter}),
\end{align}
where $F^{(k)}$ is the vector of $F^{(k)}_{ij}$ and $G^{(k)}$ is the
vector of $G^{(k)}_{ij}$, for all $(i,j) \in \mc E_{\text{span},k}$
with $i< j$. Finally, by concatenating the dynamics \eqref{eq: intra
  dynamics k} for all clusters, we obtain
\begin{align}\label{eq: intra dynamics}
  \dot x_{\text{intra}} = F(x_\text{intra}) + G(x_\text{intra},x_\text{inter}).
\end{align}

\begin{figure}[t]
  \centering 
  \subfigure[]
  {
   \includegraphics[width=0.405\columnwidth]{./Graphical_example}
  \label{fig: example x intra x inter 1}
  }
    \;\;
    \subfigure[]{
    \includegraphics[width=.47\columnwidth]{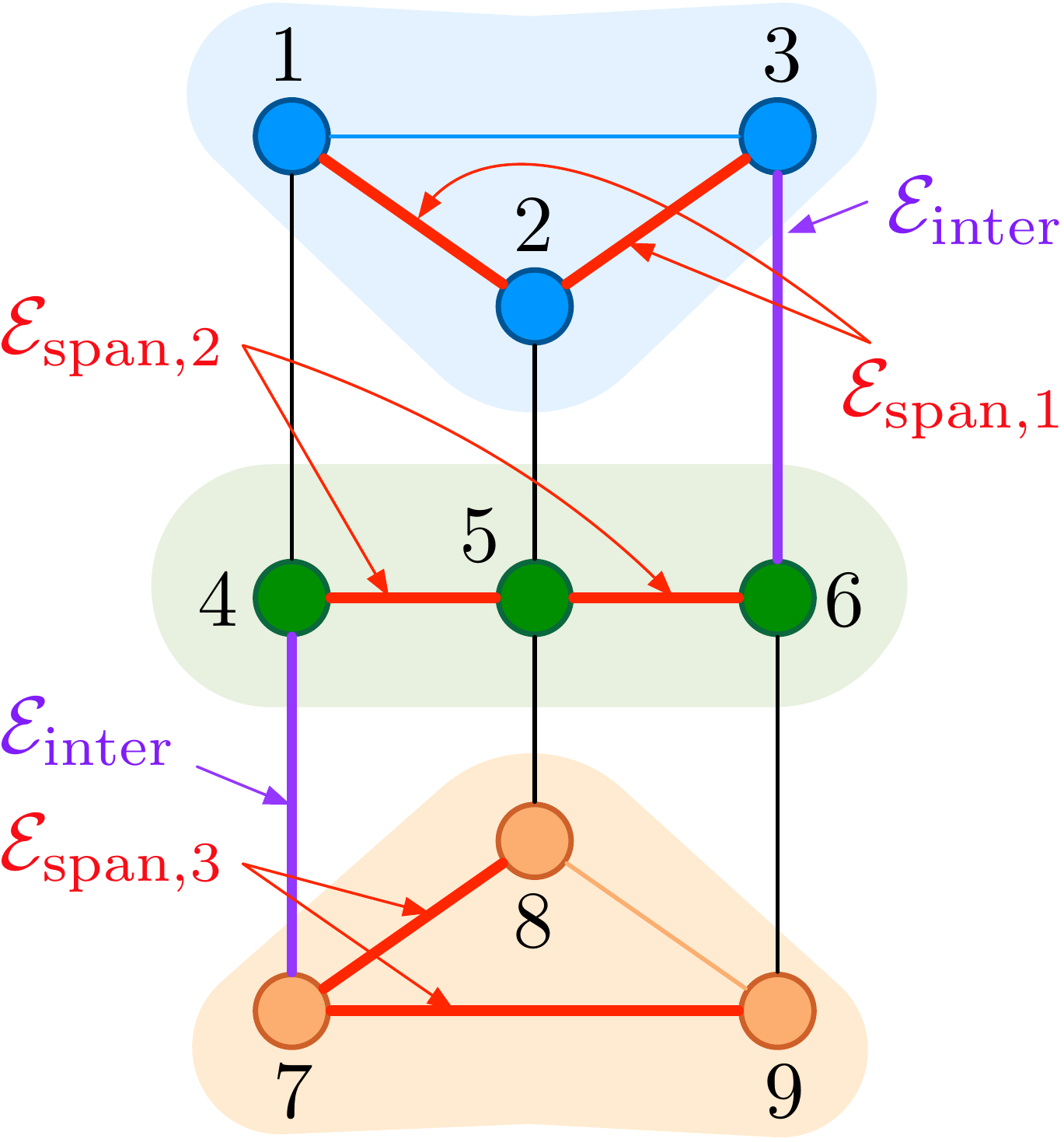}
  \label{fig: example x intra x inter 2}
  }
\caption{This figure illustrates the graph-theoretic definitions introduced in 
Section \ref{sec: section 3} for a network of $9$ Kuramoto oscillators. (see also Example \ref{ex: graphical example}).
Fig. \ref{fig: example x intra x inter 1} shows the partitions $\mc P = \{\mc P_{1}, \mc P_{2}, \mc P_{3}\}$. In Fig. \ref{fig: example x intra x inter 2}, $\mc E_{\text{span},1}$, $\mc E_{\text{span},2}$, and
  $\mc E_{\text{span},3}$ represent (in red) the edges of the intra-cluster spanning
  trees $\mc T_{1}$, $\mc T_{2}$ and $\mc T_{3}$, while the edges belonging to the set $\mc E_\text{inter}$ are depicted in purple.
}
\label{fig: example x intra x inter}
\end{figure}

\begin{example}{\bf \emph{(Illustration of the
      definitions)}}\label{ex: graphical example}
  We provide here an illustrative example of the definitions
  introduced in this section. Consider the network in Fig. \ref{fig:
    example x intra x inter 1} with partition
  $\mc P = \{\mc P_{1}, \mc P_{2}, \mc P_{3}\}$, where
  $\mc P_{1}=\{1,2,3\}$, $\mc P_{2}=\{4,5,6\}$ and
  $\mc P_{3}=\{7,8,9\}$. Fig. \ref{fig: example x intra x inter 2}
  illustrates the definitions of spanning trees, together with the
  edge sets $\mc E_{\text{span},k}$ $(k=1,2,3)$, and the inter-cluster
  edges in $\mc E_\text{inter} = \{(3,6),~(4,7) \}$.  The vectors of
  intra-cluster differences read as
  $x_{\text{intra}}^{(1)}=[x_{12}\;x_{23}]^\transpose$,
  $x_{\text{intra}}^{(2)}=[x_{45}\:x_{56}]^\transpose$, and
  $x_{\text{intra}}^{(3)}=[x_{78}\;x_{79}]^\transpose$, whereas the
  vector of inter-cluster differences reads as
  $x_\text{inter} = [x_{36}\; x_{47}]^\transpose$. 

  For the partition $\mc P_1$, order the edges as $\ell_1 =(1,2)$,
  $\ell_2 =(1,3)$, and $\ell_3 =(2,3)$. Then, a spanning tree is
  $\mc T_1 = (\mc P_1, \mc E_{\text{span},1})$, with
  $\mc E_{\text{span},1} = \{(1,2),(2,3)\}$, and the (oriented) incidence
  matrices $B_1$ of $\mc G_1$ and $B_{\text{span},1}$ of $\mc T_1$
  are
\begin{align*}
B_1 = 
\begin{bmatrix} 
-1 & -1 & 0 \\
1 & 0 & -1 \\
0 & 1 & 1
 \end{bmatrix},\;\;\;
 B_{\text{span},1} = 
\begin{bmatrix} 
-1 & 0  \\
1 & -1  \\
0 & 1 
 \end{bmatrix}.
 \end{align*}
 Finally, the matrix
 $T_{\text{intra},1} = B_{1}^\transpose
 (B_{\text{span},1}^{\transpose})^{\dagger}$ satisfies
 \begin{align*}
   T_{\text{intra},1} = 
   \begin{bmatrix} 
     1 & 1 & 0\\
     0 & 1 & 1
   \end{bmatrix}
             ^\transpose
             .
 \end{align*}
\oprocend
\end{example}

\subsection{Local asymptotic stability of $\mc S_{\mc P}$ via
  perturbation theory}\label{sec: perturbation}

In what follows we will make use of perturbation theory of dynamical systems to provide our first stability condition. We first introduce the following instrumental result.

\begin{lemma}{\bf \emph{(Properties of intra-cluster dynamics)}}\label{lemma: intra cluster dynamics}
The intra-cluster dynamics \eqref{eq: intra dynamics} satisfies the following properties:
      \begin{enumerate}
      
        \item the Jacobian matrix $J_\text{intra}$ of $F(x_\text{intra})$
    computed at the origin is Hurwitz stable and can be written as
    \begin{align}\label{eq: Jintra}
      J_\text{intra} = \left. \frac{\partial
      F(x_\text{intra})}{\partial \subscr{x}{intra}}\right|_{x_\text{intra} = 0} =
      \blkdiag \left( J_1, \dots, J_m\right), 
    \end{align}
    where, for $k \in \until{m}$, $T_{\text{intra},k}$ is as in
    \eqref{eq: Tintra} and
    \begin{align}\label{eq: J_k}
      J_k = -B_{\text{span},k}^\transpose B_k\,
        \mathrm{diag}(\{a_{ij} \}_{(i,j)\in\mc E_{k}})
        T_{\text{intra},k}.
    \end{align}
    Thus, the origin is an exponentially stable equilibrium of the
    system $\dot x_{\text{intra}} = F(x_\text{intra})$;

  \item There exist constants $\gamma^{(k\ell)} \in \real_{>0}$ such
    that
    \begin{align}\label{eq: bound on Gij}
      \| G^{(k)}({x_{\text{intra}}},x_\text{inter}) \| \le \sum_{\ell = 1}^m \gamma^{(k\ell)} \|  {x_{\text{intra}}^{(\ell)}} \|,
    \end{align}
    for all $k, \ell \in \until{m}$. Specifically, 
    \begin{align}\label{eq: gamma k ell}
      \gamma^{(k\ell)} =2 \max_{r} \; n_{\text{intra},r}\, \tilde\gamma^{(k\ell)},
    \end{align}
    where, for any $i \in \mc P_k$,
    \begin{align}\label{eq: gamma hat}
     \tilde \gamma^{(k\ell)} = \begin{cases} 
	\displaystyle \sum_{j\in \mc P_{\ell}}
        a_{ij}, &  \text{ if } \ell\ne k, \\ 
	\displaystyle \sum_{\substack{\ell=1
            \\ 
            \ell\ne k}}^{m}\sum_{j\in \mc P_{\ell}} a_{ij}, & \text{ otherwise.} 
	\end{cases}
    \end{align}

  \end{enumerate}
\end{lemma}
\smallskip

As formalized in the next theorem, Lemma \ref{lemma: intra cluster
  dynamics}, together with results on stability of perturbed systems
\cite[Chapter 9]{HKK:02}, implies that the origin of \eqref{eq: intra
  dynamics}, and thus the cluster synchronization manifold
$\mc S_{\mc P}$, is exponentially stable for some choices of the
network weights. Recall that an $M$-matrix is a real nonsingular
matrix $A = [a_{ij}]$ such that $a_{ij} \le 0$ for all $i\ne j$ and
all leading principal minors are positive \cite[Chapter
2.5]{RAH-CRJ:94}.

\begin{theorem}{\bf \emph{(Sufficient condition on network weights for
      the stability of $\mc S_{\mc P}$)}}\label{thm: stability of Sp}
  Let $\mc S_{\mc P}$ be the cluster synchronization manifold
  associated with a partition $\mc P = \{\mc P_1, \dots, \mc P_m\}$ of
  the network $\mc G$ of Kuramoto oscillators. 
  Let $\gamma^{(k\ell)}$ be the constants defined
  in \eqref{eq: gamma k ell}. Define the matrix
  $S \in \real^{m \times m}$ as
\begin{align}\label{eq: matrix S}
  S = [s_{k\ell}] = \begin{cases}
    \lambda_{\text{max}}^{-1}(X_{k})-\gamma^{(kk)}& \text{ if } k=\ell,\\
    -\gamma^{(k\ell)} & \text{ if } k\neq \ell,
  \end{cases}
\end{align}  
where  $X_{k}\succ0$ is such that $J_{k}^{\transpose}X_{k}+X_{k}J_{k}=-I$, with $J_k$ as in \eqref{eq: J_k}.
If $S$ is an $M$-matrix, then the cluster synchronization manifold
$\mc S_{\mc P}$ is locally exponentially stable.
\end{theorem}
\medskip

\begin{remark}{\bf \emph{(Family of bounds)}}\label{remark: tightest
    perturb condition}
  In \eqref{eq: matrix S}, the matrices $X_k$ can be selected as the
  solutions to the Lyapunov equations
  $J_{k}^{\transpose}X_{k}+X_{k}J_{k}=-Q_k$, for arbitrary positive
  definite matrices $Q_k$. Yet, selecting $Q_k = I$ for all $k$ yields
  a tighter stability bound. This follows because (i) if $S$ is an
  $M$-matrix, then $S+\Delta$ remains an $M$-matrix whenever $\Delta$
  is a nonnegative diagonal matrix \cite[Theorem 2.5.3]{RAH-CRJ:94},
  and (ii) the ratio
 $\lambda_{\text{min}}(Q_{k})/ \lambda_{\text{max}}(X_{k})$ is
  maximal whenever $Q_k = I$ \cite[Exercise 9.1]{HKK:02}.  \oprocend
\end{remark}

Theorem \ref{thm: stability of Sp} describes a sufficient condition on
the network weights for the stability of the cluster synchronization
manifold. Loosely speaking, the cluster synchronization manifold is
exponentially stable when the intra-cluster coupling (measured by
$\lambda_{\text{max}}^{-1}(X_{k})-\gamma^{(kk)}$) is sufficiently
stronger than the perturbation induced by the inter-cluster
connections (measured by $\gamma^{(k\ell)}$). In
  particular, the term $\lambda_{\text{max}}^{-1}(X_{k})$ is
  proportional to the intra-cluster weights and it is implicitly
  related to the network topology. In fact, the matrix $X_{k}$ is the
  solution of a Lyapunov's equation containing $J_k$, whose spectrum
  coincides with the stable eigenvalues of the negative Laplacian
  matrix of the $k$-th cluster. We refer the interested reader to the
  proof of Lemma \ref{lemma: intra cluster dynamics}. 
  Finally, we remark that a result akin to Theorem \ref{thm: stability
    of Sp} has been derived in \cite{JQ-QM-HG-YS-YK:17}, although for
  interconnected systems whose coupling functions are required to
  satisfy certain assumptions that fail to hold in the
  Kuramoto~model.

\begin{figure}[t]
\centering
   \subfigure[]
    {
\includegraphics[width=0.41\columnwidth]{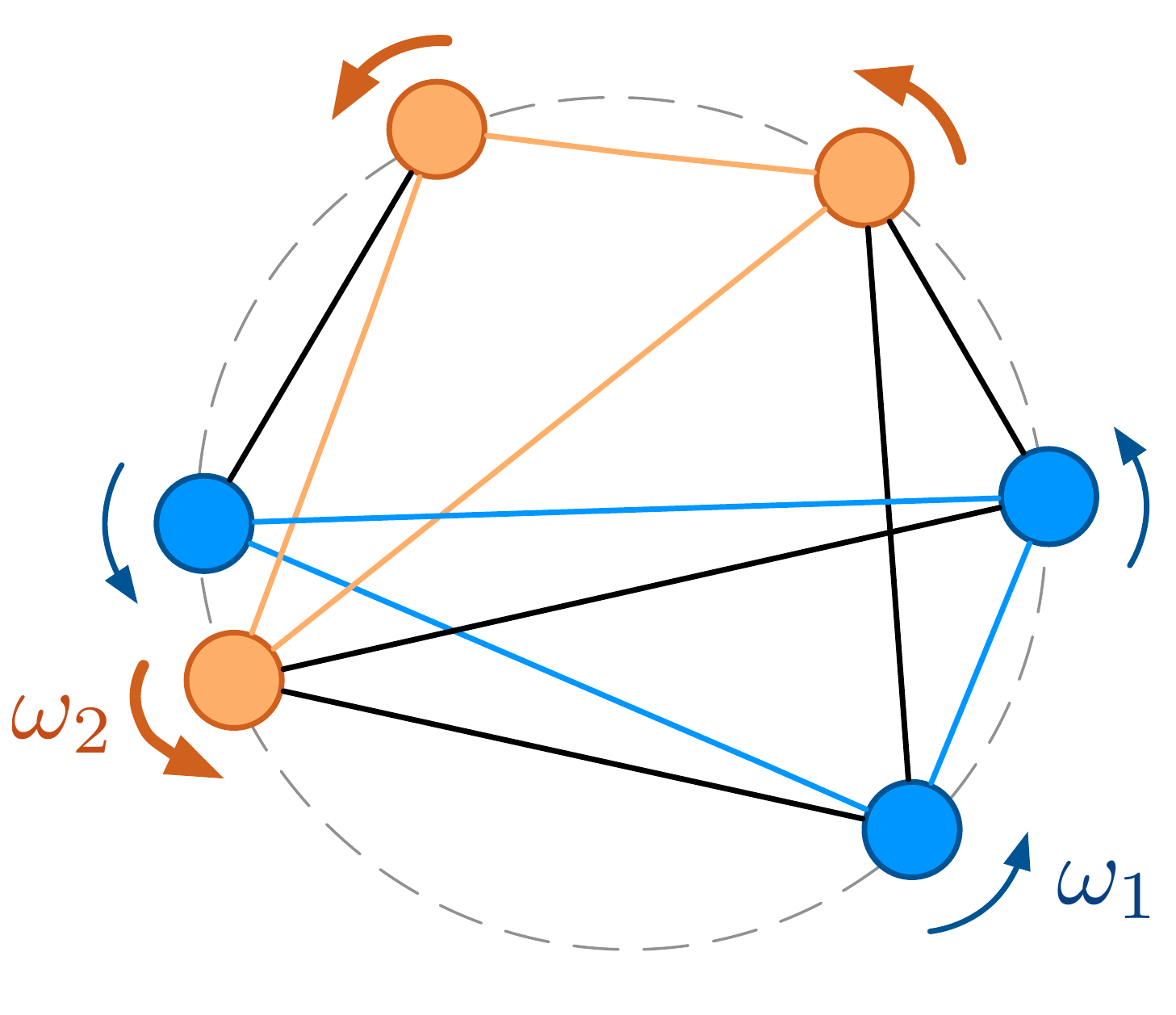}
        \label{fig: example perturbation network}
        } \vspace{0.2cm}
 \subfigure[]
    {
\includegraphics[width=0.52\columnwidth]{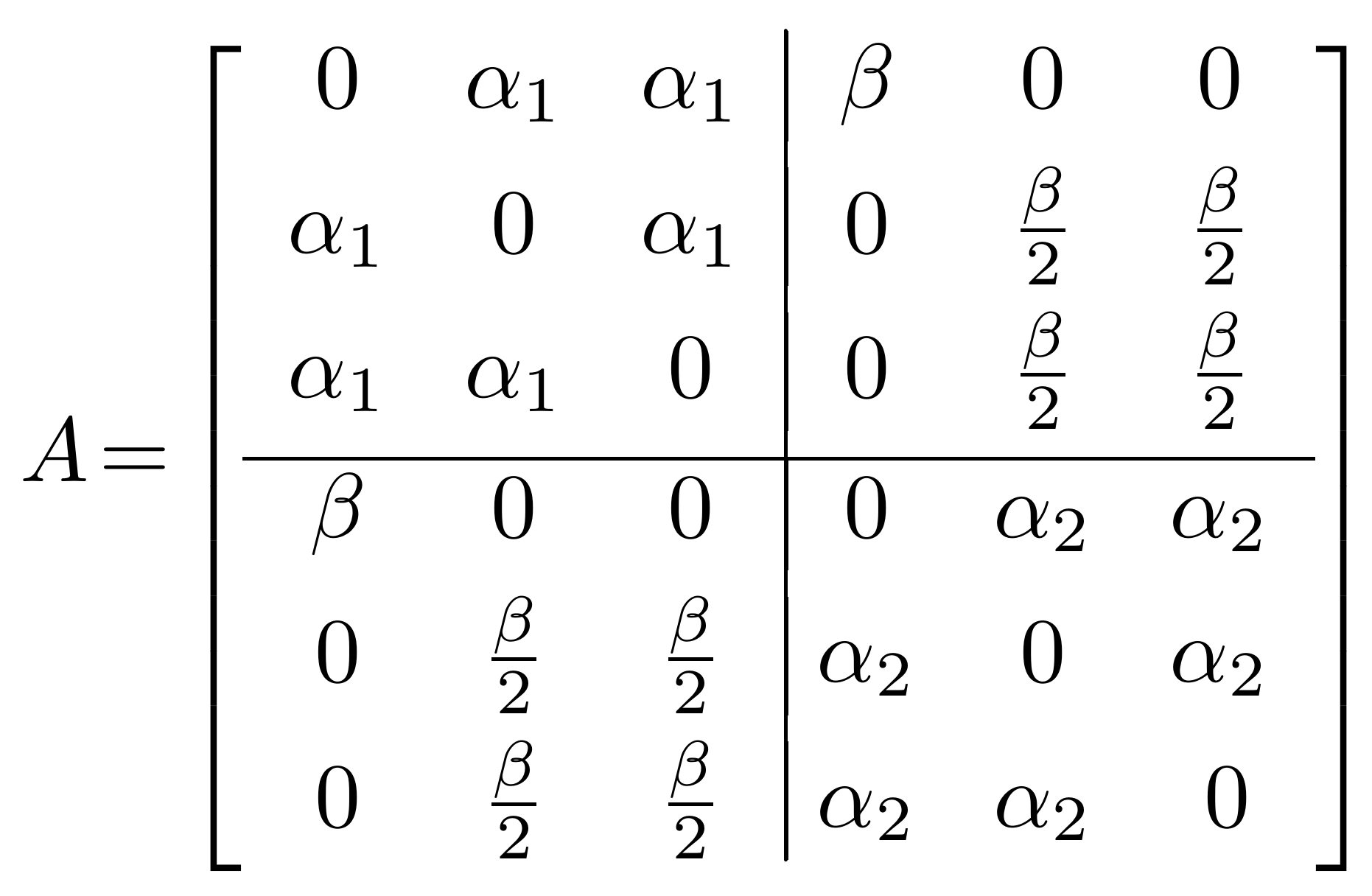}
        \label{fig: example perturbation network adjacency}
        }
          \vspace{-0.5cm}
  \includegraphics[width=.5\columnwidth]{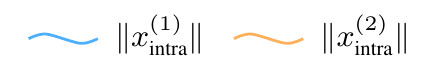}\\
\subfigure[]
    {
        \includegraphics[width=.455\columnwidth]{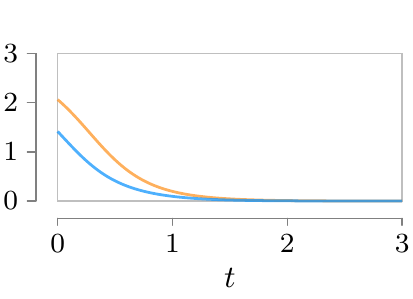}\label{fig: convergence alpha stable}
    }
    \subfigure[]
    {
    	\includegraphics[width=.455\columnwidth]{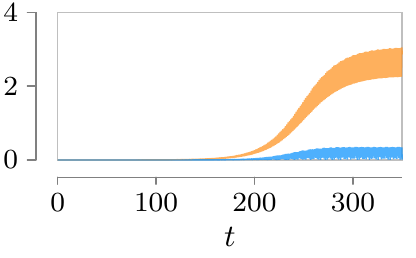}\label{fig: convergence alpha unstable}
    }

    \caption{Fig. \ref{fig: example perturbation network} illustrates the network of 6 Kuramoto oscillators in Example \ref{ex:
        perturbation thm}. We identify the clusters $\mc P_1$ and
      $\mc P_2$ in blue and orange,
      respectively.
      Fig. \ref{fig: example perturbation network adjacency} contains the adjacency matrix of the network in Fig. \ref{fig: example perturbation network}. The parameters
      $\alpha_1, \alpha_2$, and $\beta$ represent the
      intra-cluster and inter-cluster weights, respectively.
        Fig. \ref{fig: convergence alpha stable} illustrates the stability of the cluster synchronization manifold $\mc S_{\mc P}$ for
      $\alpha_{1}=\alpha_{2}=1$ and
      $\beta=0.1$, as predicted by Theorem \ref{thm: stability of Sp}.
      Fig. \ref{fig: convergence alpha unstable} shows that $\mc S_{\mc P}$ is unstable when 
       $\alpha_{1}=\beta=1$ and $\alpha_{2}=0.001$.
      }
  \label{fig: example perturbation}
\end{figure}

\begin{example}{\bf \emph{(Tradeoff between intra- and
      inter-cluster weights)}}\label{ex: perturbation thm}
  Consider the network in Fig.~\ref{fig: example perturbation network}
  with partition $\mc P = \{\mc P_1, \mc P_2 \}$, where
  $\mc P_1 = \{1,2,3\}$ and $\mc P_2 = \{4,5,6\}$, natural frequencies
  $\omega_1=1$ and
  $\omega_{2}=6$ for the oscillators in
  $\mc P_1$ and $\mc P_2$, and adjacency matrix as in Fig. \ref{fig:
    example perturbation network adjacency}.  The parameters
  $\alpha_{1},\alpha_{2}\in\real_{>0}$ and $\beta \in\real_{>0}$
  denote the strength of the intra- and inter-cluster coupling,
  respectively. Let $\alpha_{1}=\alpha_{2}$, and construct the matrix
  $S$ as in Theorem \ref{thm: stability of Sp}:
  \begin{align*}
    S = 
    \begin{bmatrix} 
      \lambda_\text{max}^{-1}(X_{1})-\gamma_{11} & -\gamma_{12} \\ 
      -\gamma_{12} & \lambda_\text{max}^{-1}(X_2)-\gamma_{22}
    \end{bmatrix},
  \end{align*}
  where $X_{k}\succ0$ is such that
  $J_{k}^{\transpose}X_{k}+X_{k}J_{k}=-I$,
 $\lambda_{\text{max}}^{-1}(X_1) = \lambda_{\text{max}}^{-1}(X_2) =
  2\,\alpha_{1}$ and, from \eqref{eq: gamma k ell},
  $\gamma_{ij} = 4\, \beta $ for all $i$, $j$. By inspecting all
  leading principal minors, $S$ is an $M$-matrix if
  $\alpha_1 /\beta > 4$, and the cluster synchronization manifold
  $\mc S_{\mc P}$ is exponentially stable (see Fig. \ref{fig:
    convergence alpha stable}). We remark that, when
  $\alpha_1 \neq \alpha_2$, the synchronization manifold
  $\mc S_{\mc P}$ can become unstable, as we verify numerically in
  Fig. \ref{fig: convergence alpha unstable}.  \oprocend
\end{example}

\begin{figure}[t]
  \centering 
  \includegraphics[width=1\columnwidth]{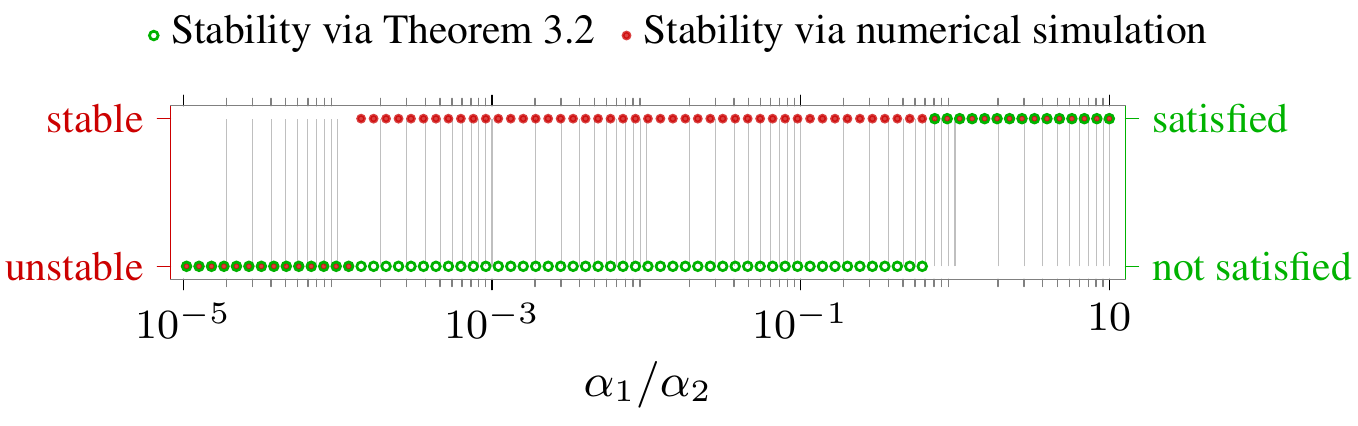}
  \caption{This Figure shows that the condition in Theorem \ref{thm:
      stability of Sp} leads to conservative stability bounds.  For
    the network in Example \ref{ex: perturbation thm}, we let
    $\beta =0.1$ and plot, as a function of the ratio
    $\alpha_1/\alpha_2$, the stable configurations predicted by
    Theorem \ref{thm: stability of Sp} (green) and those found
    numerically.   For each value of $\alpha_1/\alpha_2$,
      we assess numerical stability by making use of the Floquet
      stability theory \cite[Chapter 5]{WJR:96} and by resorting to
      statement (i) in Lemma \ref{lemma: inter cluster nom traj}. This
      is possible because the partition in Example \ref{ex:
        perturbation thm} has only two clusters.}
  \label{fig: example comparison stability}
\end{figure}

The stability condition in Theorem \ref{thm: stability of Sp} depends
only on the network weights, and typically leads to conservative
bounds (see also Fig.~\ref{fig: example comparison stability}). To
derive refined stability conditions, we next characterize how the
natural frequencies of the oscillators affect stability of the cluster
synchronization manifold.

\subsection{Local asymptotic stability of $\mc S_{\mc P}$ when the
  oscillators' natural frequencies in disjoint clusters are
  sufficiently~different}\label{sec: frequency}
Natural frequencies play a fundamental role for full and cluster
synchronization of Kuramoto oscillators. However, while heterogeneity
of the natural frequencies typically impedes full synchronization
\cite{Doerfler2014}, we will show that cluster synchronization is in
fact facilitated when the oscillators in different clusters have
sufficiently different natural frequencies. We start with an
asymptotic result that is valid for arbitrary networks and partitions,
and then improve our results for the case of partitions containing
only two clusters.

\begin{theorem}{\bf \emph{(Stability of $\mc S_{\mc P}$ for large
      natural frequency differences)}}\label{thm: stability of Sp with
    limit omega}
  Let $\mc S_{\mc P}$ be the cluster synchronization manifold
  associated with a partition $\mc P = \{\mc P_1, \dots, \mc P_m\}$ of
  the network $\mc G$ of Kuramoto oscillators. Let
  $\omega_i \in \real_{>0}$ be the natural frequency of the
  oscillators in the cluster $\mc P_i$, with $i \in \until{m}$. In the
  limit $|\omega_i - \omega_j | \rightarrow \infty$, for all
  $i,j \in \until{m}$, $i\neq j$, the cluster synchronization manifold
  $\mc S_{\mc P}$ is locally exponentially stable.
\end{theorem}
\medskip

Theorem \ref{thm: stability of Sp with limit omega} shows that
heterogeneity of the natural frequencies of the oscillators in
different clusters facilitates cluster synchronization, independently
of the network weights. We remark that a similar behavior was also
identified in \cite{CF-DSB-AC-FP:16}, although with a different method
and definition of synchronization.

We next improve upon Theorem
\ref{thm: stability of Sp with limit omega} by analyzing the case
where the natural frequencies are finite and the partition $\mc P$
contains only two clusters. To this aim, let
$\mc P = \{ \mc P_1, \mc P_2\}$ and assume, without loss of
generality, that $\omega_{2}\ge \omega_{1}$, where $\omega_i$ is the
natural frequency of the oscillators in $\mc P_i$. Define
\begin{align*}
  \bar \omega &= \omega_2 - \omega_1, \text{ and } 
  \bar a = \sum_{k \in \mc P_2} a_{ik} + \sum_{k \in \mc P_1} a_{jk} ,
\end{align*}
for any $i \in \mc P_1$ and $j \in \mc P_2$. The next result
characterizes the inter-cluster phase difference when the network
evolves on the cluster synchronization manifold.

\begin{lemma}{\bf \emph{(Nominal inter-cluster
      difference)}}\label{lemma: inter cluster nom traj}
  Let $\mc S_{\mc P}$ be the cluster synchronization manifold
  associated with a partition $\mc P = \{\mc P_1, \mc P_2\}$ of the
  network $\mc G$ of Kuramoto oscillators. Let
  $\theta (0) \in \mc S_{\mc P}$ (equivalently,
  $\subscr{x}{intra} (0) = 0$). Then, if $\subscr{x}{intra} = 0$ at all
  times and $\bar\omega > \bar a$,
  \begin{align}\label{eq: h(t)}
       x_{\text{inter}}(t) &= \begin{cases} 
h(t), &\text{ if } t \neq t_0+kT,~ k\in\mathbb{Z},\\
                \pi,& \text{ if } t=t_0+kT,~ k\in\mathbb{Z}.
\end{cases}\notag \\
      &\triangleq \subscr{x}{nom} (t),
  \end{align}
  where $$
  h(t) = 2 \tan^{-1}\!\left(\!\frac{\bar a+ \sqrt{\bar
            \omega^2- \bar a^2}\, \tan\!\left( \frac{\sqrt{ \bar \omega^2-
                \bar a^2}}{2}( t+ \tau)\right)}{\bar \omega}\!\right),$$
                $t_0 = -\tau + \pi/\sqrt{ \bar \omega^2-
                \bar a^2}$, $T = 2\pi/\sqrt{\bar \omega^2-\bar a^2}$, and 
                 $\tau \in \real$ is a constant that depends only on
  $\theta(0)$.  Moreover,
  \begin{enumerate} 
  \item $x_\text{nom}$ is $T$-periodic with zero time average, and

    \item the following inequality holds:
      \begin{align}\label{eq: bound on cos(h) nom}
        \left| \int_0^t \cos( x_\text{nom}(\tau)) \,\dif \tau \right|
        \le \frac{1}{\bar a}  \log\left(\frac{\bar \omega + \bar
        a}{\bar  \omega - \bar a} \right).
      \end{align}
    \end{enumerate}
\end{lemma}
\smallskip
 

\begin{remark}{\bf \emph{(Constant versus time-varying inter-cluster
      difference)}}\label{remark: condition omega}
  The values of $\bar \omega$ and $\bar a$ determine the behavior of
  the inter-cluster phase difference. In particular, if
  $\bar \omega < \bar a$, then the inter-cluster difference evolves as
  in \eqref{eq: h(t)}.\footnote{In fact, $\sqrt{\bar \omega^2-\bar a^2}$ becomes a complex number and, by recalling that $\tan( \mathsf{i} \alpha) =\mathsf{i}\tanh(\alpha)$, where $\alpha\in\real$, in \eqref{eq: h(t)} we have $x_\text{inter}(t) = 2 \tan^{-1}((\bar a-\sqrt{\bar a^2-\bar \omega^2}\,\tanh(\sqrt{\bar a^2-\bar \omega^2}(t+\tau)/2))/\bar \omega)$.}
If $\bar \omega = \bar a$, \eqref{eq: differences dynamics} reduces to
  $\dot x_{\text{inter}} = \bar a - \bar a \sin(
  x_{\text{inter}})$, which can be integrated:
  \begin{align}
   \bar at  \notag &=  \int_{x_\text{inter}(0)}^{x_\text{inter}(t)}  (1-  \sin(s))^{-1}\mathrm{d} s \\
\bar  at &= \frac{2\sin(x_\text{inter}(t)/2)}{\cos(x_\text{inter}(t)/2)-\sin(x_\text{inter}(t)/2)}
      +  \tau. \label{eq: temp intergal}
  \end{align}
  By substitution, it can be verified that
  \begin{align}
    x_\text{inter}(t) &= 2 \cos^{-1}\!\left(\!\frac{\bar a t - \tau
                        +2}{\sqrt{2(\bar a t  - \tau +1)^{2}+2 }}\!\right)\!, \notag
  \end{align}
  satisfies equation \eqref{eq: temp intergal}.
In both cases ($\bar \omega \le \bar a$), $x_\text{inter}$ converges
to the constant value
$2\tan^{-1}((\bar a-\sqrt{\bar a^2-\bar \omega^2})/\bar \omega)$ as
$t$ increases to infinity. In other words, if
$\bar \omega \le \bar a$, then the phases of the oscillators in the
two clusters evolve with the same frequency, and the oscillators are
phase locked (see Fig.~\ref{fig: stable 1} and \cite[Remark
1]{Doerfler2014}). Instead, if $\bar \omega > \bar a$, the clusters
evolve with different frequencies, and the inter-cluster phase
difference follows a limit cycle (see Fig.~\ref{fig: stable 2} and
\cite[Chapter 2]{HKK:02}). \oprocend
\end{remark}

\begin{figure}[t]
  \centering 
  
  \vspace{0.05cm}
  
  \includegraphics[width=.6\columnwidth]{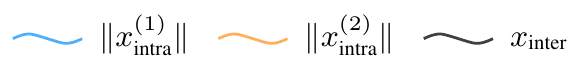}
  
    \subfigure[]{
    \includegraphics[width=.46\columnwidth]{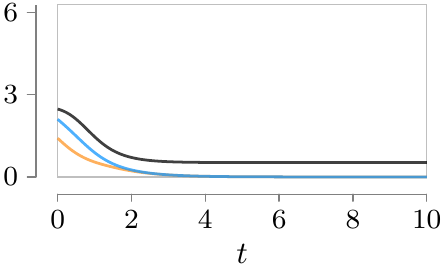}
  \label{fig: stable 1}} 
  \subfigure[]{
    \includegraphics[width=.46\columnwidth]{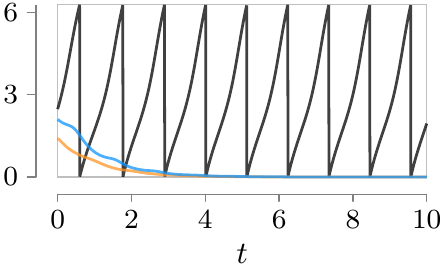}
  \label{fig: stable 2}} 
\caption{
    For the network in Example \ref{ex: perturbation
    thm} with $\alpha_1=\alpha_2 = \beta=1$, $\bar{a}=2$ and $\bar\omega=1$, Fig. \ref{fig: stable 1} shows that the clusters are synchronized (as $\Vert x_{\text{intra}}^{(1)}\Vert$ and $\Vert x_{\text{intra}}^{(2)}\Vert$ converge to zero), yet all oscillators remain phase locked ($x_\text{inter}$ converges to a constant).  Instead, Fig. \ref{fig: stable 2} shows that the inter-cluster difference follows a limit cycle when $\alpha_1=\alpha_2 = \beta=1$, $\bar{a}=2$ and $\bar\omega=6$.
}
\label{fig: clusters or phase locking}
\end{figure}

In the remainder of this section we assume that
$\bar \omega > \bar a$, so that the clusters evolve with different
frequencies (see Remark \ref{remark: condition omega}).
Leveraging Lemma \ref{lemma: inter cluster nom traj}, we next present
a refined condition for the stability of the cluster synchronization
manifold.

\begin{theorem}{\bf \emph{(Sufficient condition on network weights and
      natural frequencies for the stability of
      $\mc S_{\mc P}$)}}\label{thm: stability based on Lyapunov}
  Let $\mc S_{\mc P}$ be the cluster synchronization manifold
  associated with a partition $\mc P = \{\mc P_1, \mc P_2\}$ of the
  network $\mc G$ of Kuramoto oscillators. Let
  $\omega_i \in \real_{>0}$ be the natural frequency of the
  oscillators in the cluster $\mc P_i$, with $i \in \{1,2\}$. Let
  $\subscr{J}{intra}$ be as in Lemma \ref{lemma: intra cluster
    dynamics}, and
  $J_\text{inter}=\partial G(x_\text{intra}, x_\text{inter})
  / \partial x_\text{intra}$ along the trajectory $x_\text{intra} = 0$
  and $x_\text{inter} = x_\text{nom}$. The cluster synchronization
  manifold $\mc S_{\mc P}$ is locally exponentially stable if the
  following inequality holds:
  \begin{align}\label{eq: bound frequency}
    \left(\frac{\bar \omega + \bar a}{\bar \omega -  \bar a}
    \right)^{\frac{2}{\bar a} \|J_{\text{inter}}\|}  < 1+\frac{1}{2
    \lambda_{\text{max}}(X) \|J_{\text{intra}}\|},
  \end{align}
  where $X\succ 0$ is the solution of
  $J_\text{intra}^\transpose X + X J_\text{intra} = -I$.
\end{theorem}
\medskip

Theorem \ref{thm: stability based on Lyapunov} provides a quantitative
condition on the network weights and the natural frequencies of the
oscillators to ensure stability of the cluster synchronization
manifold. It can be shown that (i) when the inter-cluster weights
decrease to zero ($\bar a \rightarrow 0$) and $\bar \omega$ remains
bounded, then $\|\subscr{J}{inter}\|/\bar a$ remains bounded, the
left-hand side of \eqref{eq: bound frequency} converges to $1$, and
the inequality is automatically satisfied, and (ii) when $\bar \omega$
grows ($\bar \omega \rightarrow \infty$) and the inter-cluster weights
remain bounded, the left-hand side of \eqref{eq: bound frequency}
converges to $1$ and the inequality is automatically satisfied. The
role of the intra-cluster connections on the stability of
$\mc S_{\mc P}$ cannot be evaluated directly from \eqref{eq: bound
  frequency} because of the dependency of the right-hand side on
$\lambda_\text{max} (X)$. The following result, however, suggests that
the synchronization manifold may remain exponentially stable when the
intra-cluster weights are homogeneous, independently of the
inter-cluster weights and the natural frequencies.

\begin{theorem}{\bf \emph{(Stability of $\mc S_{\mc P}$ with
      homogeneous clusters)}}\label{corollary: stability with
    homogeneous clusters and large natural frequencies}
  Let $\mc S_{\mc P}$ be the cluster synchronization manifold
  associated with a partition $\mc P = \{\mc P_1, \mc P_2\}$ of the
  network $\mc G$ of Kuramoto oscillators. Let
  $\omega_i \in \real_{>0}$ be the natural frequency of the
  oscillators in the cluster $\mc P_i$, with $i \in \{1,2\}$. If
  $\subscr{J}{intra} = \alpha I$, for some constant
  $\alpha \in \real_{<0}$, then the cluster synchronization manifold
  $\mc S_{\mc P}$ is locally exponentially stable.
\end{theorem}
\medskip

We provide an example that illustrates the
stability conditions derived in Theorem~\ref{thm: stability based on
  Lyapunov}.

\begin{figure}[t]
  \centering 
  \subfigure[]{
    \includegraphics[width=.45\columnwidth]{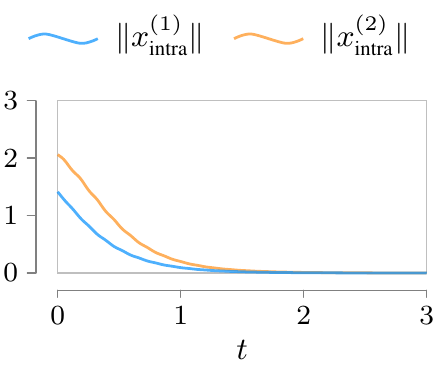}
    \label{fig: stability frequencies example}}
  \subfigure[]{
    \includegraphics[width=.46\columnwidth]{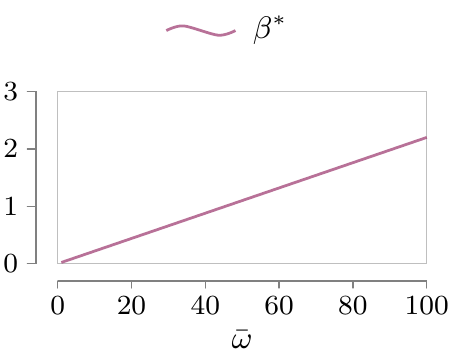}
    \label{fig: beta omega bar}} 
  \caption{For the network in Example \ref{ex: perturbation thm}, Fig. \ref{fig: stability frequencies example} illustrates the stability of $\mc S_{\mc P}$ when $\alpha_{1}=\alpha_{2}=\beta= \omega_1 = 1$ and $\omega_2=47$, as predicted by the condition in Theorem \ref{thm: stability based on Lyapunov}. For the same network and weights, Fig. \ref{fig: beta omega bar} shows the largest value of inter-cluster weights $\beta^*$ that satisfies \eqref{eq:
      bound frequency} with equality. As predicted by Theorem  \ref{thm: stability of Sp with
    limit omega} and Theorem \ref{thm: stability based on Lyapunov}, stability of the cluster synchronization manifold $\mc S_{\mc P}$ is preserved when $\bar \omega$ grows with the inter-cluster weights.
      }
\end{figure}

\begin{figure}[t]
  \centering 
  \includegraphics[width=1\columnwidth]{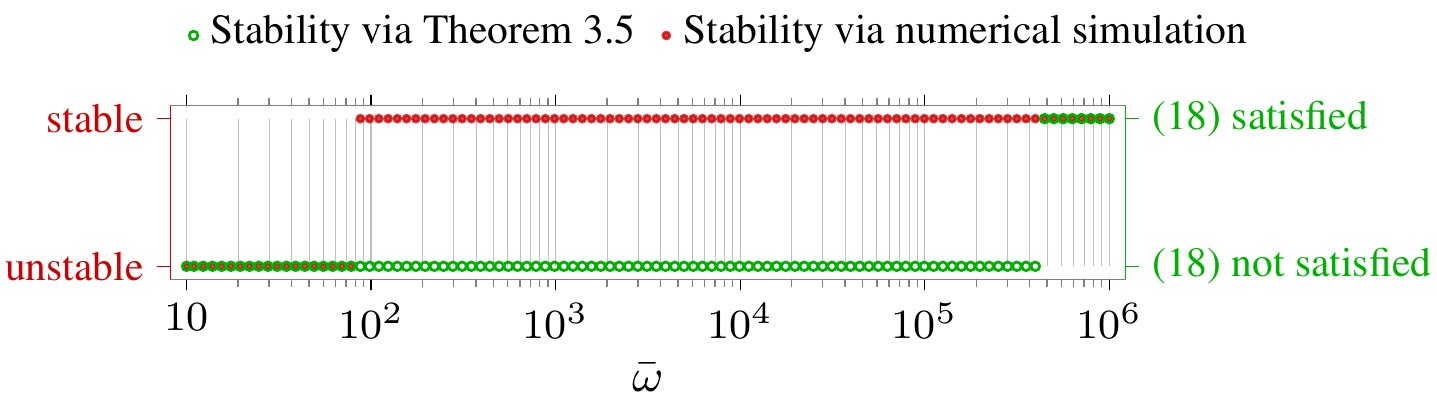}
  \caption{For the network in Example \ref{ex: perturbation thm}, we
    let $\alpha_{1}=\beta=1$ and $\alpha_{2}=10^{-4}$ and plot, as a
    function of $\bar \omega$, the stable configurations predicted by
    Theorem \ref{thm: stability based on Lyapunov} (green) and those
    found numerically.  For each value of $\bar \omega$,
      we assess numerical stability (in red) by making use of the
      Floquet stability theory \cite[Chapter 5]{WJR:96} and by
      resorting to statement (i) in Lemma \ref{lemma: inter cluster
        nom traj}. This is possible because the partition in Example
      \ref{ex: perturbation thm} contains two clusters. Although
    condition \eqref{eq: bound frequency} is conservative, it captures
    the effect of large $\bar \omega$ on the stability of
    $\mc S_{\mc P}$.  }
  \label{fig: example comparison stability frequency}
\end{figure}

\begin{example}{\bf \emph{(Heterogeneity of natural frequencies
      improves stability of the cluster synchronization manifold)}}\label{ex:
    enforcing stability synchronization manifold}
  Consider the network of Kuramoto oscillators in Example \ref{ex:
    perturbation thm}. Fig. \ref{fig: stability frequencies example} illustrates that the
  cluster synchronization manifold is asymptotically stable when the
  condition in Theorem \ref{thm: stability based on Lyapunov} is
  satisfied. Fig. \ref{fig: beta omega bar} illustrates the tradeoff
  in the latter stability condition between the natural frequency
  $\bar \omega$ and the inter-cluster strength measured by
  $\beta^{*}$, which denotes the largest inter-cluster weight
  $\beta$ (see Example \ref{ex: perturbation thm}) such that
  \eqref{eq: bound frequency} is still satisfied. Further, we show in
  Fig. \ref{fig: example comparison stability frequency} that, while
  being conservative, condition \eqref{eq: bound frequency} captures
  the fact that stability of the cluster synchronization manifold can
  be recovered by increasing $\bar \omega$. Namely, choosing
  the same network weights that yield instability as in Fig. \ref{fig: convergence alpha unstable},
  we show that stability of the cluster synchronization manifold is recovered as the difference in natural frequencies grows.
  \oprocend
\end{example}

We conclude this section with a discussion of cluster synchronization
in asymmetric networks and identical nodes.

\begin{remark}{\bf \emph{(Extension to networks with asymmetric
      weights)}}\label{remark: asymmetric}
Symmetry of the network weights is typically exploited to provide
conditions for the stability of the full synchronization manifold in
networks of Kuramoto oscillators \cite{Doerfler2014}. We rely on the
symmetry assumption (A1) to derive statement (i) in Lemma \ref{lemma:
  intra cluster dynamics}, which supports our main theorems. However,
these results remain valid for bidirected
  graphs,\footnote{A bidirected graph is a directed graph where
    $(i,j) \in \mc E$ implies $(j,i) \in \mc E$. The adjacency matrix
    of a bidirected graph needs not be symmetric.} 
  provided that the Jacobian
  $J_{\text{intra}}$ can be proven to be Hurwitz. In other words,
  Assumption (A1) is used to guarantee stability of the isolated
  clusters, and not of the cluster configuration.
  \oprocend
\end{remark}

\begin{remark}{\bf \emph{(Cluster synchronization in networks of
      identical oscillators)}}
  This paper focuses on heterogeneous oscillators and leverages
  mismatches in the natural frequencies and the network weights to
  characterize the stability of the cluster synchronization
  manifold. Yet, cluster synchronization can also arise in networks of
  homogeneous Kuramoto oscillators, where all units have equal natural
  frequencies and all edges have equal weight (e.g., see
  Fig. \ref{fig: identical frequency}). With the exception of Theorem
  \ref{thm: stability of Sp with limit omega}, which is also
  applicable in the case of identical edge weights, our stability
  results cannot predict cluster synchronization in networks of
  identical oscillators, a question that we leave as the subject of
  future investigation.
  \oprocend
\end{remark}

\begin{figure}[t]
  \centering 
  \subfigure[]{
    \includegraphics[width=.38\columnwidth]{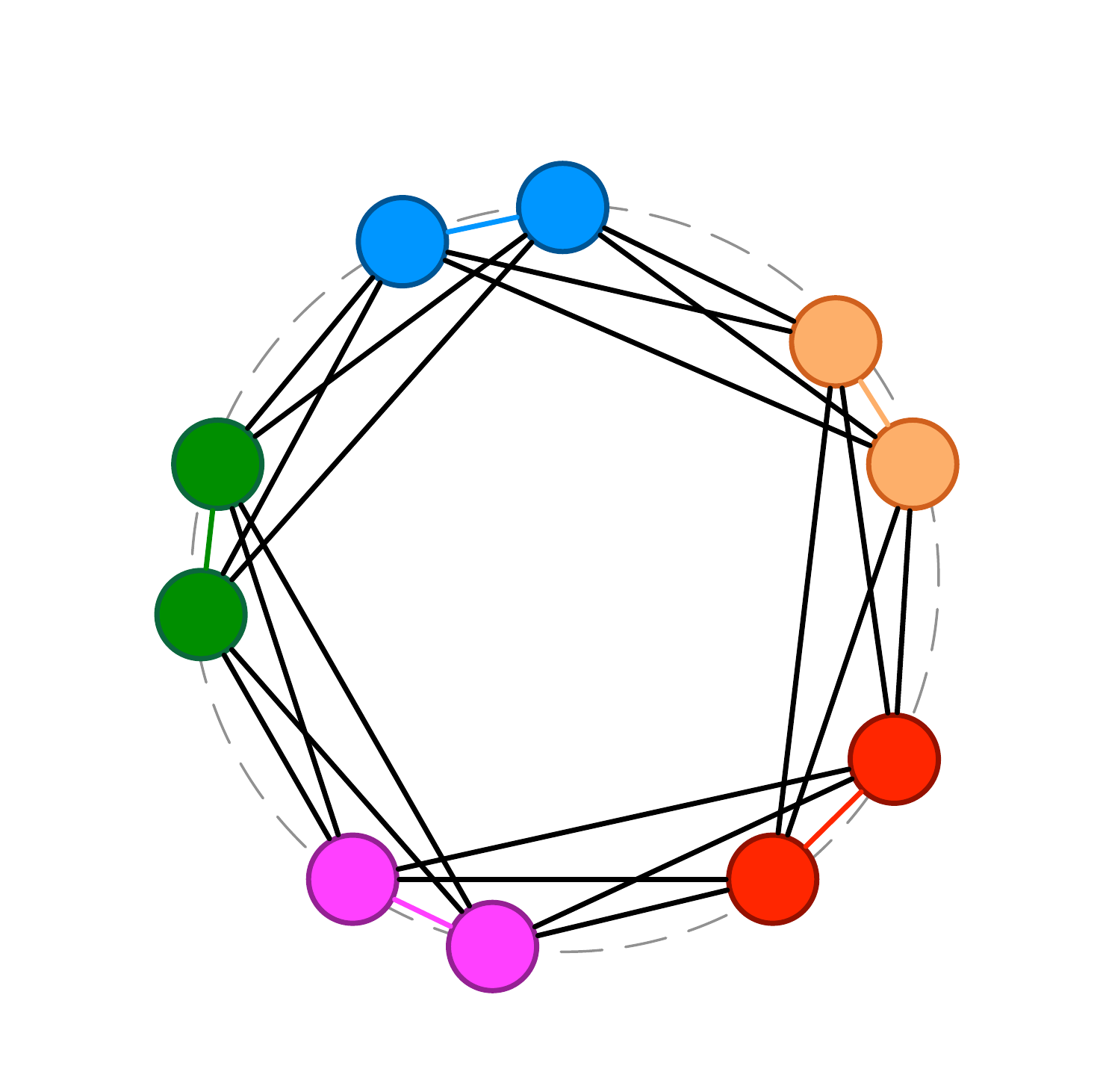}
    \label{fig: network identical frequency}}\;\;
  \subfigure[]{
    \includegraphics[width=.46\columnwidth]{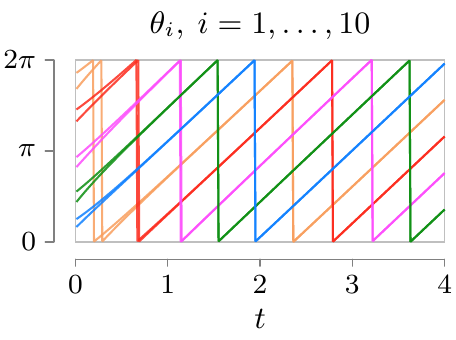}
    \label{fig: simulation identical frequency}} 
  \caption{Fig. \ref{fig: network identical frequency}
      illustrates a network of $10$ Kuramoto oscillators with
      partition
      $\mc P = \{\mc P_1, \mc P_2, \mc P_3, \mc P_4, \mc P_5 \}$,
      where each cluster is color-coded. All oscillators have
      identical natural frequency $\omega = 3$ and all edges have
      unit weight. As illustrated in Fig. \ref{fig: simulation
        identical frequency}, the cluster synchronization manifold
      associated to $\mc P$ is stable, showing that cluster
      synchronization is possible even in networks of identical
      Kuramoto oscillators with identical edge weights.}\label{fig:
    identical frequency}
\end{figure}

\section{Conclusion and future work}\label{sec: conclusion} 
In this work we characterize conditions for the stability of cluster
synchronization in networks of oscillators with Kuramoto dynamics,
where multiple synchronized groups of oscillators coexist in a
connected network. We derive conditions showing that the cluster
synchronization manifold is locally exponentially stable when (i) the
intra-cluster coupling is sufficiently stronger that the inter-cluster
coupling, (ii) the differences of natural frequencies of the
oscillators in disjoint clusters are sufficiently large, or, (iii) in
the case of two clusters, if the intra-cluster dynamics is
homogeneous. To the best of our knowledge, our results are the first
to characterize the stability of the cluster synchronization manifold
in sparse and weighted networks of heterogeneous Kuramoto oscillators.

Directions of future research include the characterization of tighter
stability bounds, the design of methods to control the formation of
time-varying synchronized clusters, and the extension of Theorem
\ref{thm: stability based on Lyapunov} to an arbitrary number of
clusters.

\begin{appendix}
  In this section we provide the proofs of the results presented in
  Section \ref{sec: section 3}, together with some instrumental
  lemmas.
  
  \subsection{Proofs of the results in Section \ref{sec: perturbation}}
  \begin{pfof}{Lemma \ref{lemma: intra cluster dynamics}}
    \emph{Proof of statement (i)}. Notice that the block-diagonal form
    of the Jacobian matrix $J_\text{intra}$ follows
    directly from the form of $F(x_{\text{intra}})$ in \eqref{eq:
      intra dynamics}. Therefore, the stability of $J_\text{intra}$ is
    equivalent to the stability of the diagonal blocks
    $J_{k}$. Let $\theta^{(k)}$ be the vector of
      $\theta_i$, $i\in\mathcal{P}_k$ and, by Assumption (A2), let
      $\omega_k$ be the natural frequency of any oscillator in
      $\mc P_k$. From \eqref{eq: kuramoto}, we
      write the phase dynamics of the $k$-th cluster as (see \cite{AJ-NM-MB:04})
    \begin{align*}
      \dot \theta^{(k)} = \omega_k \1- B_{k}\, \mathrm{diag}(\{a_{ij}
      \}_{(i,j)\in\mc E_{k}}) \sin (B_{k}^\transpose\theta^{(k)}) .
    \end{align*}
  Because the phase differences satisfy
  $ x_{\text{intra}}^{(k)} = B_{\text{span},k}^\transpose
  \theta^{(k)}$ and $x^{(k)} = B_{k}^\transpose\theta^{(k)}$, we have
     \begin{equation}\label{eq: equivalent xspan}
       \dot x_{\text{intra}}^{(k)} = - B_{\text{span},k}^\transpose
       B_{k}\, \mathrm{diag}(\{a_{ij} \}_{(i,j)\in\mc E_{k}})  \sin (x^{(k)}),
     \end{equation} 
     where we have used the property
       $B_{\text{span},k}^\transpose \1=0$.  Using \eqref{eq:
       Tintra}, the Jacobian matrix of \eqref{eq: equivalent xspan}
     computed at $x_{\text{intra}}^{(k)} = 0$ reads as
     \begin{align}\label{eq: Jk}
       J_k= - B_{\text{span},k}^\transpose B_{k}\,
       \mathrm{diag}(\{a_{ij} \}_{(i,j)\in\mc E_{k}})
       T_{\text{intra},k} . 
     \end{align}
     Recall that the Laplacian matrix of the graph
       $\mc G _k$ satisfies
     \begin{align*}
       \mc L_{\mc G_{k}} = B_{k}\,\mathrm{diag}(\{a_{ij} \}_{(i,j)\in\mc
       E_{k}}) B_{k}^\transpose ,
     \end{align*}
     and that, because $\mc G_k$ is connected, the eigenvalues of
     $-\mc L_{\mc G_{k}}$ have negative real part, except one single
     eigenvalue located at the origin with eigenvector $\1$. Define
     the matrix $W_k = [B_{\text{span},k} \; \1]^\transpose$ and
     notice that, because $B_{\text{span},k}^\transpose \1 = 0$ and
     $B_{\text{span},k}$ being full column rank \cite[Theorem
     8.3.1]{Godsil2001}, then $W_k$ is invertible and
     $W_k^{-1} = [(B_{\text{span},k}^\transpose)^{\dag} \;
     (\1^\transpose)^{\dag}]$. Therefore we have
     \begin{align*}
       W_k (-\mc L_{\mc G_{k}}) W_k^{-1} =
       \begin{bmatrix}
         J_k & 0\\
         0 & 0
       \end{bmatrix}
             ,
     \end{align*}
     where we have used that
     $T_{\text{intra},k} = B_k^\transpose
     (B_{\text{span},k}^\transpose)^\dag$ in \eqref{eq: Jk}. This
     shows that $J_k$ contains only the stable eigenvalues of
     $-\mc L_{\mc G_{k}}$.

  \noindent\emph{Proof of statement (ii)}. Notice that, for any
  $(j,z) \in \mc E$ with $j \in \mc P_k$, $z \in \mc P_\ell$, and
  $k \neq \ell$, the difference $\mathrm{dif}\mathrm{f}( \mathrm{p}(j,z))$ in
  $G_{ij}^{(k)}(x_\text{inter},x_\text{inter})$ in equation \eqref{eq:
    Gij} can be rewritten as
  \begin{align*}
    \mathrm{dif}\mathrm{f}( \mathrm{p}(j,z)) = \mathrm{dif}\mathrm{f}(\mathrm{p}(j,k^{*}))\! +\! \mathrm{dif}\mathrm{f}(\mathrm{p}(k^{*}\!,\ell^{*}))
    \!+ \mathrm{dif}\mathrm{f}(\mathrm{p}(\ell^{*}\!,z)),
  \end{align*}
   where $k^*$ and $\ell^*$ are such that
    $\mathrm{p}(k^{*},\ell^{*})$ is the shortest path on $\mc T$
    connecting the clusters $\mc P_k$ and $\mc P_\ell$.  Then,
  \begin{align*}
    &G_{ij}^{(k)}(x_\text{inter},x_\text{inter}) = \\
    &\!\sum_{\substack{\ell =1\\ \ell\ne k}}^{m}\!\sum_{z \in\mc P_\ell}
    \!\! \left[a_{jz} \!\sin(\mathrm{dif}\mathrm{f}(\mathrm{p}(j,k^{*})) \!+\! \mathrm{dif}\mathrm{f}(\mathrm{p}(k^{*}\!,\ell^{*}))
    \!+\! \mathrm{dif}\mathrm{f}(\mathrm{p}(\ell^{*}\!,z))) \right.\\
    &\left.\;\;\;\;\; -
      a_{iz}\sin(\mathrm{dif}\mathrm{f}(\mathrm{p}(i,k^{*}))+\mathrm{dif}\mathrm{f}(\mathrm{p}(k^{*}\!,\ell^{*}))\!
      + \mathrm{dif}\mathrm{f}(\mathrm{p}(\ell^{*}\!,z)))\right]\!.
  \end{align*}
Notice that $\mathrm{dif}\mathrm{f}(\mathrm{p}(i,k^{*}))$ and
$\mathrm{dif}\mathrm{f}(\mathrm{p}(j,k^{*}))$ contain only differences in
$x_{\text{intra}}^{(k)}$, and $\mathrm{dif}\mathrm{f}(\mathrm{p}(\ell^{*},z))$
only differences in~$x_{\text{intra}}^{(\ell)}$.

\noindent Notice that $\sin(a + b) = \sin(a) + \delta$, with
$|\delta| \le |b|$.\footnote{ Letting
    $\delta = \sin(a+b)-\sin(a)$, we have
    $ |\delta| = |2\sin(\frac{b}{2})\cos(a+\frac{b}{2})| \le
    |2\sin(\frac{b}{2})|$, from which the inequality
    $|\delta| \le |b|$ follows.} Then,
  \begin{align}\label{eq: g ell z 2}
    G_{ij}^{(k)} (x_\text{intra},x_\text{inter})  =& \sum_{\substack{\ell
                                                  =1\\ \ell\ne k}}^{m} \sum_{z \in\mc P_\ell} [a_{jz} \left(
    \sin(\mathrm{dif}\mathrm{f}(\mathrm{p}(k^{*},\ell^{*})) +\delta_{jz} \right) \notag \\
    &- a_{iz} \left(
      \sin(\mathrm{dif}\mathrm{f}(\mathrm{p}(k^{*},\ell^{*}))+\delta_{iz} \right)] \notag\\
      =& \!\sum_{\substack{\ell =1\\ \ell\ne k}}^{m}\!\left(\sum_{z \in\mc P_\ell} [(a_{jz}-a_{iz}) \!\sin(\mathrm{dif}\mathrm{f}(\mathrm{p}(k^{*},\ell^{*})))]\right.
     \notag \\
      &\left.+ \sum_{z \in\mc P_\ell} [a_{jz}\delta_{jz}- a_{iz}\delta_{iz}]\right) \notag \\
      \overset{\text{(A3)}}{=}&\sum_{\substack{\ell
                                                  =1\\ \ell\ne k}}^{m}\sum_{z\in\mc P_{\ell}}
      [a_{jz}\delta_{jz}\!-\! a_{iz}\delta_{iz}], \notag
  \end{align}
  where $\delta_{jz}$ and $\delta_{iz}$ are upper bounded by
  $\sqrt{n_{\text{intra},k}} \|x_{\text{intra}}^{(k)}\| +
  \sqrt{n_{\text{intra},\ell}}
  \|x_{\text{intra}}^{(\ell)}\|$. Therefore, we have the following
  bound:
  \begin{align}
    |G_{ij}^{(k)}| &\le\,  \sum_{\substack{\ell =1 \\  \ell \neq k }}^m \left( \sum_{z\in\mc P_{\ell}} a_{jz}|\delta_{jz}| + \sum_{z\in\mc P_{\ell}}  a_{iz}|\delta_{iz}| \right)\notag\\
    &\overset{\text{(A3)}}{\le} 2 \sum_{\substack{\ell =1 \\  \ell \neq k }}^m \sum_{z\in\mc P_{\ell}} a_{jz}\! \left(\sqrt{n_{\text{intra},k}}\|x_{\text{intra}}^{(k)}\|
    +\sqrt{n_{\text{intra},\ell}}\|x_{\text{intra}}^{(\ell)}\|\right) \notag \\
    &= 2 \sum_{\ell = 1}^m  \sqrt{n_{\text{intra},\ell}}\,\, \tilde\gamma^{(k\ell)}_{ij} \|  {x_{\text{intra}}^{(\ell)}} \|,\notag
  \end{align}
  where
  \begin{align*}
    \tilde \gamma^{(k\ell)}_{ij} = \begin{cases}
      \displaystyle \sum_{\substack{\ell=1 \\ \ell\neq k}}^m \sum_{z\in\mc P_{\ell}} a_{jz}, &\text{ if } \ell=k,  \\
      \displaystyle  \sum_{z\in\mc P_{\ell}} a_{jz}, & \text{ otherwise.}
    \end{cases}
  \end{align*}
  To conclude,
  $\| G^{(k)} \| \le \sqrt{n_{\text{intra},k}} \max_{(i,j)\in\mc
    E_{\text{span},k}} \vert G_{ij}^{(k)}\vert$, and, due to (A3),
  $\tilde \gamma^{(k\ell)}_{ij} = \tilde \gamma^{(k\ell)}$ is
  independent of $i$ and $j$. Thus,
  \begin{align}
    \| G^{(k)} \| \le \sum_{\ell = 1}^m2 \max_{r}\, n_{\text{intra},r}
    \, \tilde \gamma^{(k\ell)} \, \|  {x_{\text{intra}}^{(\ell)}} \|. \notag
  \end{align}
This concludes the proof.
\end{pfof}

\begin{pfof}{Theorem \ref{thm: stability of Sp}} 
  The system \eqref{eq: intra dynamics} can be viewed as the
  perturbation via $G(\subscr{x}{intra}, \subscr{x}{inter})$ of
  $\subscr{\dot x}{intra} = F(\subscr{x}{intra})$, which describes the
  dynamics of $m$ disjoint networks of oscillators:
  \begin{equation}\label{eq: interconnected sys}
    \dot  x_{\text{intra}}^{(k)} = F^{(k)}({x_{\text{intra}}^{(k)}}).
  \end{equation}
  The origin of each system \eqref{eq: interconnected sys} is an
  exponentially stable equilibrium, which can be shown with the
  Lyapunov candidate
  \begin{equation*}
    V_k(x_\text{intra}) = 
    {x_{\text{intra}}^{(k)\transpose}}P_{k} {x_{\text{intra}}^{(k)}},
  \end{equation*} 
  where $P_{k}\succ 0$ is such that
  $J_{k}^{\transpose}P_{k}+P_{k}J_{k} = -Q_{k}$ for $Q_{k}\succ 0$. In
  fact, the derivative of $V$ along the trajectories \eqref{eq:
    interconnected sys} is
  \begin{align}\label{eq: Vdot subsystem}
    \dot V_k({x_{\text{intra}}^{(k)}}) &= 
                                         F^{(k)\transpose}({
                                         x_{\text{intra}}^{(k)}})
                                         P_{k}{x_{\text{intra}}^{(k)}}
                                         +
                                         x_{\text{intra}}^{(k)\transpose} P_{k}F^{(k)}({
                                         x_{\text{intra}}^{(k)}})\notag \\
                                       &= x_{\text{intra}}^{(k)\transpose} (J_{k}^{\transpose}P_{k}+P_{k}J_{k})x_{\text{intra}}^{(k)} + O(\|{x_{\text{intra}}^{(k)}}\|^3),
  \end{align}
  and the latter is strictly negative when
  $\| \subscr{x}{intra}^{(k)}\| \le r$ and $r\in \real_{>0}$ is
  sufficiently small. Further, it holds that: (i)
  $\|\partial V_{k}/\partial x_{\text{intra}}^{(k)}\|\le
  2 \lambda_{\text{max}}(P_{k}) \|x_{\text{intra}}^{(k)}\|$, (ii)
  $\dot{V}_{k}({x_{\text{intra}}^{(k)}}) \le
  -\lambda_{\text{min}}(Q_{k}) \|x_{\text{intra}}^{(k)}\|^{2}$, and
  (iii) the perturbation terms
  $G^{(k)}({x_{\text{intra}}}, x_\text{inter})$ are linearly bounded
  in $\|{x_{\text{intra}}^{(k)}}\|$ following statement (ii) in Lemma
  \ref{lemma: intra cluster dynamics}. 

  Consider now the following Lyapunov candidate for \eqref{eq: intra
    dynamics}:
  \begin{equation*}
    V(x_\text{intra}) = \sum_{k = 1}^m d_k
    V_k({x_{\text{intra}}^{(k)}}), \;\;\; d_k>0 .
  \end{equation*}
  From \cite[Chapter 9.5]{HKK:02} we have:
  \begin{equation}\label{eq: }
    \dot V(x_\text{intra}) \le -\frac{1}{2} (DS+S^\transpose D) \| x_\text{intra} \|^2,
\end{equation} 
where $D = \mathrm{diag}(d_1, \dots, d_m)$, and $S$ satisfies
\begin{align}\label{eq: matrix S temp}
  S = [s_{k\ell}] = \begin{cases}
    \frac{\lambda_{\text{min}}(Q_{k})}{\lambda_{\text{max}}(P_{k})}-\gamma^{(kk)}& \text{ if } k=\ell,\\
    -\gamma^{(k\ell)} & \text{ if } k\neq \ell.
  \end{cases}
\end{align}
The origin of \eqref{eq: intra dynamics} is locally exponentially
stable if $S$ is an $M$-matrix \cite[Lemma 9.7 and Theorem
9.2]{HKK:02}. Finally, choosing $Q_k = I$ in \eqref{eq: matrix S temp} yields condition \eqref{eq: matrix S} in Theorem \ref{thm: stability of Sp}.
\end{pfof}

\subsection{Proofs of the results in Section \ref{sec: frequency}}
Let $\mc C$ be the set of connected clusters pairs, that is,
\begin{align*}
  \mc C = \setdef{(\ell, z)}{ \exists\; (i,j) \in \mc E \text{ with } i\in \mc
  P_\ell,  j\in \mc P_z, \text{ and } \ell <z}.
\end{align*}
With a slight abuse of notation, for any $(\ell, z) \in \mc C$, we define
$x^{(\ell z)} = x_{ij}$, for any node $i \in \mc P_\ell$ and
$j \in \mc P_z$.
\begin{lemma}{\bf \emph{(Linearized intra-cluster
      dynamics)}}\label{lemma: linearized intra cluster}
  The linearization of the intra-cluster dynamics \eqref{eq: intra
    dynamics} around the trajectory $x_\text{intra} = 0$ and
  $x_\text{inter} = x_\text{nom}$ reads as follows:
  \begin{align}\label{eq: linear part}
    \dot x_\text{intra} = \left(  J_\text{intra} + \subscr{J}{inter}\right) x_\text{intra},
  \end{align}
  where $J_{\text{intra}}$ is defined in Lemma \ref{lemma: intra
    cluster dynamics}, and
  \begin{align*}
    \subscr{J}{inter} = \left.\frac{\partial G}{\partial x_{\text{intra}}} \right
    \vert_{\shortstack[l]{ $\scriptstyle x_{\text{intra}}=0$ \\
    $\scriptstyle  x_{\text{inter}} =  x_{\text{nom}}$}} \triangleq \sum_{(\ell, z)\in \mc C} \cos( x^{\,(\ell z)})
    \, J_{\text{inter}}^{(\ell z)} .
  \end{align*}
\end{lemma}
\begin{proof}
  Linearization of \eqref{eq: intra dynamics} around the trajectory
  $(x_\text{intra}, x_\text{inter})= (0, x_\text{nom})$ yields
  $\partial F / \partial x_\text{intra} = J_\text{intra}$ and
  $\partial G / \partial x_\text{intra} = J_\text{inter}$. The
  remaining derivatives vanish. That is,
  $\partial F / \partial x_\text{inter} = 0$ because $F$ does not
  depend on $x_\text{inter}$, and
  $\partial G / \partial x_\text{inter} = 0$ because of Assumption
  (A3). In fact, for any intra-cluster difference $x_{ij}$ with
  $i,j\in\mc P_{\ell}$, $\ell\in\until{m}$, 
\begin{equation*}
  \left.\frac{\partial G_{ij}}{\partial x_{\text{inter}}} \right
    \vert_{\shortstack[l]{ $\scriptstyle x_{\text{intra}}=0$ \\
    $\scriptstyle  x_{\text{inter}} =  x_{\text{nom}}$}} = \sum_{(\ell, z)\in \mc C} \cos( x^{\,(\ell z)}) \underbrace{\sum_{k\in  \mc P_{z}}
  [a_{jk} - a_{ik}]}_{=0} = 0  .
\end{equation*} 
This concludes the proof.
\end{proof}

We next characterize an asymptotic property of the inter-cluster
differences through the following instrumental result.

  \begin{lemma}{\bf \emph{(Asymptotic behavior of the inter-cluster
        dynamics for large frequency differences)}}\label{lemma: limit
      of omega}
  Let $i\in\mc P_{\ell}$, $j \in \mc P_{z}$, and $\ell \neq z$. Then,
  the inter-cluster difference $x_{ij}$ satisfies
  \begin{align}
    \lim_{|\omega_j - \omega_i| \rightarrow \infty }\; \frac{x_{ij} (t)}{\omega_j - \omega_i} =
     t.
  \end{align}
\end{lemma}
\begin{proof}
 Let $\bar \omega_{ij} = \omega_j-\omega_i$. We rewrite \eqref{eq:
  differences dynamics} as
\begin{align}\label{eq: x_inter alternative}
  \dot x_{ij} =&\, \bar \omega_{ij} - (a_{ij}+a_{ji}) \sin( x_{ij}) \notag \\
               &+\sum_{\substack{k\ne i,j}}\left[
                 a_{jk}\sin( x_{jk})-
                 a_{ik}\sin( x_{ik} )\right].
\end{align}
From \eqref{eq: x_inter alternative}, let $\beta =\sum_{\substack{k\ne i,j}}
                 [a_{jk}+
                 a_{ik}]$, and
\begin{align}
	\dot {\underline{x}}_{\,ij} &= \bar \omega_{ij} - (a_{ij}+a_{ji}) \sin(\underline{x}_{\,ij}) - \beta, \label{eq: under x inter}\\
	\dot {\overline{x}}_{ij} &= \bar \omega_{ij} - (a_{ij}+a_{ji}) \sin(\overline{x}_{ij}) + \beta, \label{eq: over x inter}
\end{align}
with ${\underline{x}}_{\,ij}(0) = {\overline{x}}_{ij}(0)={x}_{ij}(0)$.
Integrating \eqref{eq: under x inter} yields
\begin{equation}\label{eq: intergals y t}
\int_{{x}_{ij}(0)}^{{\underline{x}}_{\,ij}(t)} \frac{\dif y}{\bar \omega_{ij} -(a_{ij}+a_{ji})\sin(y)-\beta} = \int_0^t\dif \tau.
\end{equation}
As $|\bar \omega_{ij}|$ grows, it holds that $\vert (a_{ij}+a_{ji}) + \beta \vert < |\bar \omega_{ij}|$. Therefore,
\begin{align*}
\frac{1}{\bar \omega_{ij} -(a_{ij}+a_{ji})\sin(y)-\beta}  = \frac{1}{\bar \omega_{ij}}\left[\frac{1}{1-\frac{(a_{ij}+a_{ji})\sin(y)+\beta}{\bar \omega_{ij}}}\right] \\ = \frac{1}{\bar \omega_{ij}} \sum_{k=0}^\infty \left[\frac{(a_{ij}+a_{ji})\sin(y)+\beta}{\bar\omega_{ij}}\right]^{k}.
\end{align*}
In view of the latter equality, \eqref{eq: intergals y t} becomes
\begin{align*}
t = &\,\frac{{\underline{x}}_{\,ij}(t)-{x}_{ij}(0)}{\bar \omega_{ij}} \\
&+ \frac{1}{\bar \omega_{ij}} \underbrace{\int_{{x}_{ij}(0)}^{{\underline{x}}_{\,ij}(t)} \sum_{k=1}^\infty \left[\frac{(a_{ij}+a_{ji})\sin(y)+\beta}{\bar\omega_{ij}}\right]^k \!\dif y}_{O\left(\bar \omega_{ij}^{-1} \right)},
\end{align*}
or, equivalently,
\begin{align}\label{eq: x underline}
{\underline{x}}_{\,ij}(t) =\, \bar \omega_{ij} \, t + {x}_{ij}(0)
+ O\left(\bar \omega_{ij}^{-1} \right).
\end{align}
Similarly, the solution of \eqref{eq: over x inter} has the form in
\eqref{eq: x underline}. Finally, using the Comparison Principle
\cite[Lemma 3.4]{HKK:02}, it holds that
$\underline{x}_{\,ij}(t) \le x_{ij}(t) \le \overline{x}_{ij}(t)$ for
all $t\ge 0$. Hence, $\frac{x_{ij}(t)}{\bar \omega_{ij}} \to t$
as $|\bar \omega_{ij}|\to \infty$ and this concludes the proof.
\end{proof}

We are now ready to prove Theorem \ref{thm: stability of Sp with limit
  omega}.

\begin{pfof}{Theorem \ref{thm: stability of Sp with limit omega}}
  Consider the Lyapunov candidate
  $V(\subscr{x}{intra},t) = \subscr{x}{intra}^\transpose \Gamma(t)
  \subscr{x}{intra}$, and notice that, using \eqref{eq: linear part},
  \begin{align}\label{eq: Vdot}
    &\dot V(\subscr{x}{intra},t) =\; \dot x_\text{intra}^\transpose \Gamma x_\text{intra} + x_\text{intra}^\transpose \Gamma
                   \dot x_\text{intra} + x_\text{intra}^\transpose \dot \Gamma x_\text{intra} \notag\\
                   &\hspace{.5cm} = x_\text{intra}^\transpose \!
                   \left[\vphantom{\sum_{(\ell, z)\in \mc C}} \,
                   J_\text{intra}^\transpose  \Gamma + \Gamma
                   J_\text{intra} + \dot  \Gamma \right.\notag\\
                 &\hspace{0.4cm}\left.+\!\!  \!\sum_{(\ell, z)\in \mc C}\!\!\cos ( x^{\,(\ell
                   z)} )\! \left( J_{\text{inter}}^{(\ell
                   z)\transpose} \Gamma +  \Gamma
                   J_{\text{inter}}^{(\ell z)}\right)\! \right]\!x_\text{intra}\!+\! O(\|x_\text{intra}\|^3).
  \end{align}
  Let
  \begin{align}\label{eq: Gamma dot}
    \dot \Gamma = -\sum_{(\ell, z)\in \mc C} \cos (
    x^{\,(\ell z)} )\!\left( J_{\text{inter}}^{(\ell z)\transpose} \Gamma + \Gamma
    J_{\text{inter}}^{(\ell z)} \right).
  \end{align}
  When the inter-cluster natural frequencies satisfy
  $|\omega_i- \omega_{j}| \rightarrow \infty$ for all $i,j$, then
  $\Gamma (t) \rightarrow \Gamma(0)$ for all times~$t$. In fact,
  integrating both sides of \eqref{eq: Gamma dot} and substituting
  $\Gamma(t) = \Gamma(0)$~yields
  \begin{align*}
    \int_0^t \! \dot \Gamma\dif \tau \!&= \Gamma(t) - \Gamma(0) = \Gamma (0) -
                           \Gamma (0) = 0 \\ &=\! - \!\!\sum_{(\ell, z)\in \mc C} \int_0^t
                                                                            \cos(x^{(\ell z)}) \left( J_{\text{inter}}^{(\ell
                                                                            z)\transpose} \Gamma +  \Gamma
                                                                            J_{\text{inter}}^{(\ell z)}\right)\dif \tau\\
                     &=\! - \!\!\sum_{(\ell, z)\in \mc C} \!\!\left( J_{\text{inter}}^{(\ell
                       z)\transpose} \Gamma(0) +  \Gamma(0)
                       J_{\text{inter}}^{(\ell z)}\right)\! \int_0^t\!
                       \cos(x^{(\ell z)})\dif \tau,
  \end{align*}
  which holds true because $\int \cos(x^{(\ell z)})\dif \tau = 0$ due to Lemma
  \ref{lemma: limit of omega}. Because $\subscr{J}{intra}$ is stable,
  we conclude that, when the inter-cluster natural frequencies satisfy
  $|\omega_i- \omega_{j}| \rightarrow \infty$ for all $i,j$,
  $\dot \Gamma = 0$, and there exists $\Gamma(0)$ such that \eqref{eq:
    Vdot} is strictly negative. This concludes the proof of the
  claimed statement.
\end{pfof}

\begin{pfof}{Lemma \ref{lemma: inter cluster nom traj}}
   When $\subscr{x}{intra} = 0$, the differential equation \eqref{eq:
     x_inter alternative} reduces to
   $\dot {x}_\text{inter} = \bar \omega - \bar a \sin( x_\text{inter})$,
   which is a separable differential equation with solution as in
   \eqref{eq: h(t)}.  To show that the period of \eqref{eq: h(t)} is
   equal to $T = 2\pi/\sqrt{\bar \omega^2-\bar a^2}$, we assume,
   without loss of generality, that $\tau = 0$. It is easy to see that,
   because $\tan(t)$ is $\pi$-periodic,
   $x_\text{nom}(t)= x_\text{nom}(t + 2\pi/\sqrt{\bar \omega^2-\bar
     a^2})$.  Further, notice that the variable substitution
   $z = x_\text{nom}$ in $\int_0^t \cos( x_\text{nom}) \,\dif \tau$
   yields
\begin{align}\label{eq: integral cos equality}
\int_0^t \cos( x_\text{nom}(\tau)) \,\dif \tau &= \int_{ x_\text{nom}(0)}^{ x_\text{nom}(t)} \frac{\cos(z)}{\bar \omega - \bar a \sin(z)} \,\dif z \nonumber \\
& = \frac{1}{\bar a} \log\left(\frac{\bar \omega - \bar a \sin(x(0))}{\bar \omega - \bar a \sin( x_\text{nom}(t))} \right),
\end{align}
which implies the bound \eqref{eq: bound on cos(h) nom}. To prove that
$\cos( x_\text{nom})$ has zero time average, it suffices to substitute
$t=T$ in \eqref{eq: integral cos equality}.
\end{pfof}

\begin{pfof}{Theorem \ref{thm: stability based on Lyapunov}}
Consider the Lyapunov candidate
  $V(\subscr{x}{intra},t) = \subscr{x}{intra}^\transpose \Gamma(t)
  \subscr{x}{intra}$, and notice that, using \eqref{eq: linear part},\begin{align}
\dot V(x_\text{intra},t) 
 =&\; x_\text{intra}^\transpose [\, J_\text{intra}^\transpose \Gamma + \Gamma J_\text{intra} + \dot \Gamma  \notag\\
   +& \cos( x_{\text{nom}})( J_\text{inter}^\transpose \Gamma + \Gamma J_\text{inter})  \,]x_\text{intra}+ O(\|x_\text{intra}\|^3). \label{eq:Vdot}
\end{align}
Let
$\dot \Gamma = -\cos( x_{\text{nom}})( J_\text{inter}^\transpose
\Gamma + \Gamma J_\text{inter})$ and notice that, following
\cite[Exercise 3.9 and Property 4.2]{WJR:96}, its solution satisfies

 \begin{align*}
\Gamma(t) =&  \exp\left[{ - \int_0^t \!\cos( x_{\text{nom}}(\tau))\,J_\text{inter}^\transpose\, \dif \tau }\right] \Gamma(0)\\
& \,\cdot \exp\left[{ - \int_0^t  \!\cos( x_{\text{nom}}(\tau)) J_\text{inter}\, \dif \tau}\right]\!.
\end{align*}
This implies that  $V(x_\text{intra},t)$ is a Lyapunov function for \eqref{eq: linear part} because, by Lemma \ref{lemma: inter cluster nom traj}, $ \int_0^t \cos( x_{\text{nom}}(\tau))\dif \tau$ is bounded.
Furthermore, notice that
\begin{align*}
& \exp\left[{ - \int_0^t \!\cos( x_{\text{nom}}(\tau))\,J_\text{inter}^\transpose\, \dif \tau }\right] \\
& = I + \underbrace{\sum_{k=1}^\infty\frac{(J_\text{inter}^\transpose)^k}{k!}\left( - \int_0^t \!\cos( x_{\text{nom}}(\tau))\, \dif \tau\right)^k}_{\Delta}.
\end{align*}
Thus, \eqref{eq:Vdot} can equivalently be written as
$
\dot V = x_\text{intra}^\transpose [ J_\text{intra}^\transpose \Gamma(0) \!+\! \Gamma(0) J_\text{intra} \!+\! M \,]x_\text{intra}\!+ O(\|x_\text{intra}\|^3),
$
where 
$
M = J_\text{intra}^\transpose \Delta \Gamma(0) \Delta^\transpose + \Delta \Gamma(0) \Delta^\transpose J_\text{intra} + J_\text{intra}^\transpose (\Delta \Gamma(0)+\Gamma(0) \Delta) + (\Delta \Gamma(0)+\Gamma(0) \Delta) J_\text{intra}.
$
using the triangle inequality and Lemma \ref{lemma: inter cluster nom
  traj}, we obtain
\begin{align*}
\|\Delta\|&=\left\|\sum_{k=1}^\infty\frac{(J_\text{inter}^\transpose)^k}{k!}\left( - \int_0^t \!\cos( x_{\text{nom}}(\tau))\, \dif \tau\right)^k\right\| \\
&\le \sum_{k=1}^\infty\frac{\|J_\text{inter}\|^k}{k!}\left| \int_0^t \!\cos( x_{\text{nom}}(\tau))\, \dif \tau\right|^k \\
 &= e^{ \left|\int_0^t \!\cos( x_{\text{nom}}(\tau))\, \dif \tau\right| \|J_\text{inter} \|}-1\le e^{\frac{1}{\bar a} \log\left(\frac{\bar \omega + \bar a}{\bar \omega - \bar a} \right) \|J_\text{inter} \|}-1.
\end{align*}
Because $J_\text{intra}$ is stable, there always exists $\Gamma(0)\succ 0$ such that $J_\text{intra}^\transpose \Gamma(0) + \Gamma(0) J_\text{intra} = -Q$ for any $Q\succ 0$. Thus,
\begin{align}\label{eq:Vdotsim}
\dot V \le & \; (-\lambda_{\text{min}}(Q) + \|M\| )\|x_\text{intra}\|^{2} +O(\|x_\text{intra}\|^3).
\end{align}
By a simple Lyapunov argument, the cluster synchronization manifold $\mc S_{\mc P}$ is locally exponentially stable
if $\|M\|< \lambda_{\text{min}}(Q)$.
In addition, $\|M\|$ can be upper bounded as
\begin{align*}
\| M \| \le& 2 \|J_{\text{intra}}\| \|\Gamma(0)\| \|\Delta\|(\|\Delta\|+2)\\
\le& 2 \lambda_{\text{max}}(\Gamma(0)) \|J_{\text{intra}}\|\left(e^{\frac{2}{\bar a} \log\left(\frac{\bar \omega + \bar a}{\bar \omega - \bar a} \right)\|J_\text{inter} \|}-1\right).
\end{align*}
Thus, a sufficient condition for local exponential stability is
\begin{align*}
2 \lambda_{\text{max}}(\Gamma(0)) \|J_{\text{intra}}\| \left(e^{\frac{2}{\bar a} \log\left(\frac{\bar \omega + \bar a}{\bar \omega - \bar a} \right)\|J_\text{inter} \|}-1\right)< \lambda_{\text{min}}(Q),
\end{align*}
and because the ratio $\lambda_\text{min}(Q)/\lambda_\text{max}(\Gamma(0))$ is maximized for $Q=I$ \cite[Exercise 9.1]{HKK:02}, we have
\begin{align*}
2 \lambda_{\text{max}}(\Gamma(0))\|J_{\text{intra}}\| \left(e^{\frac{2}{\bar a} \log\left(\frac{\bar \omega + \bar a}{\bar \omega - \bar a} \right)\|J_\text{inter} \|}-1\right)< 1,
\end{align*}
from which condition \eqref{eq: bound frequency} follows.
\end{pfof}

\begin{pfof}{Theorem \ref{corollary: stability with homogeneous clusters and large natural frequencies}}
From \eqref{eq:Vdot} and for $\beta \in \real_{>0}$ we have
$
\dot V(x_\text{intra},t) = \; x_\text{intra}^\transpose [\, J_\text{intra}^\transpose \Gamma + \Gamma J_\text{intra}  \,]x_\text{intra}+ O(\|x_\text{intra}\|^3)
  = \; -\beta x_\text{intra}^\transpose\Gamma x_\text{intra} + O(\|x_\text{intra}\|^3),
$
which is negative in a small neighborhood of the origin.
\end{pfof}

\end{appendix}

\bibliographystyle{unsrt}
\bibliography{BIB}

\begin{IEEEbiography}[{\includegraphics[width=1in,height=1.25in,clip,keepaspectratio]{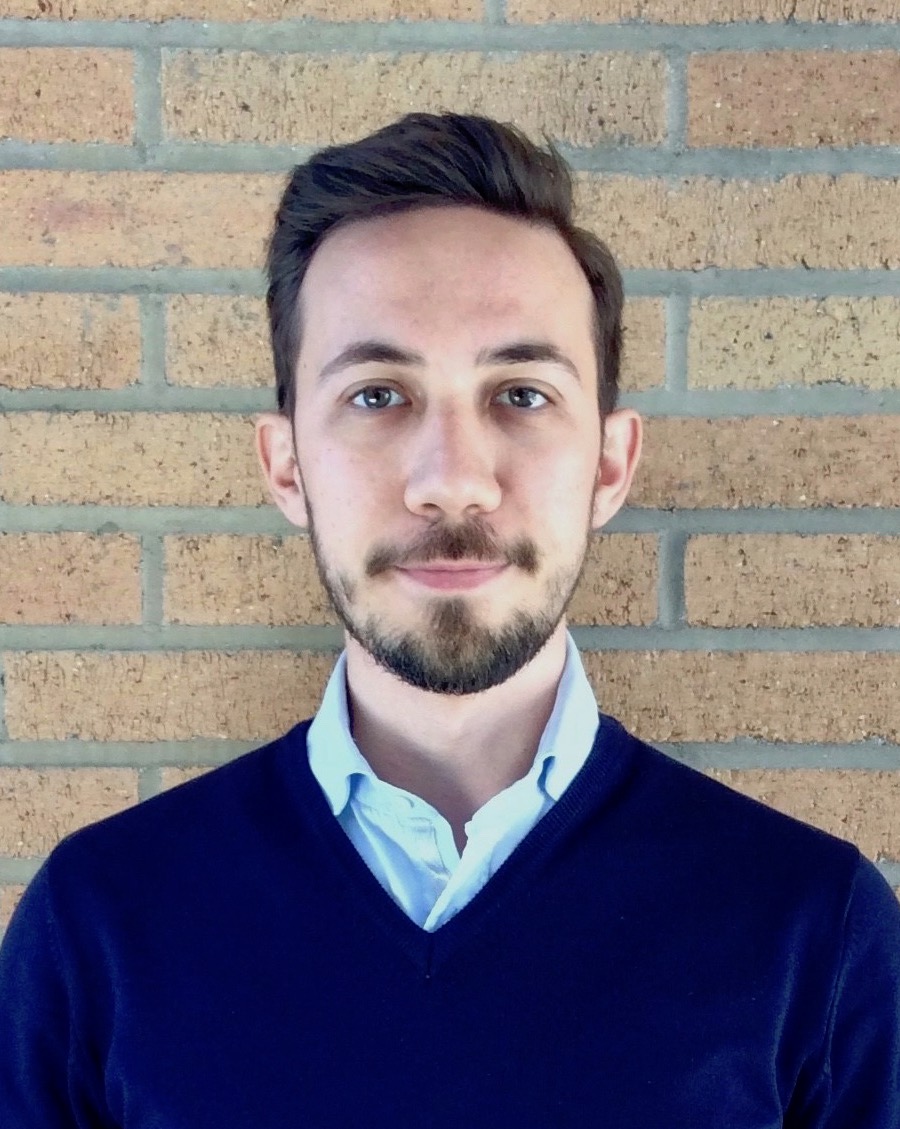}}]{Tommaso Menara}  is a PhD candidate in the Department of Mechanical Engineering at University of California at Riverside. He completed the Laurea Magistrale degree (M.Sc. equivalent) in robotics and automation engineering from the University of Pisa, Italy, in 2016, and the Laurea degree (B.Sc. equivalent) in mechatronics engineering from the University of Padova, Italy, in 2013.
His research interests include control of complex networks and network neuroscience.
\end{IEEEbiography}

\begin{IEEEbiography}%
  [{\includegraphics[width=1in,height=1.25in,keepaspectratio]{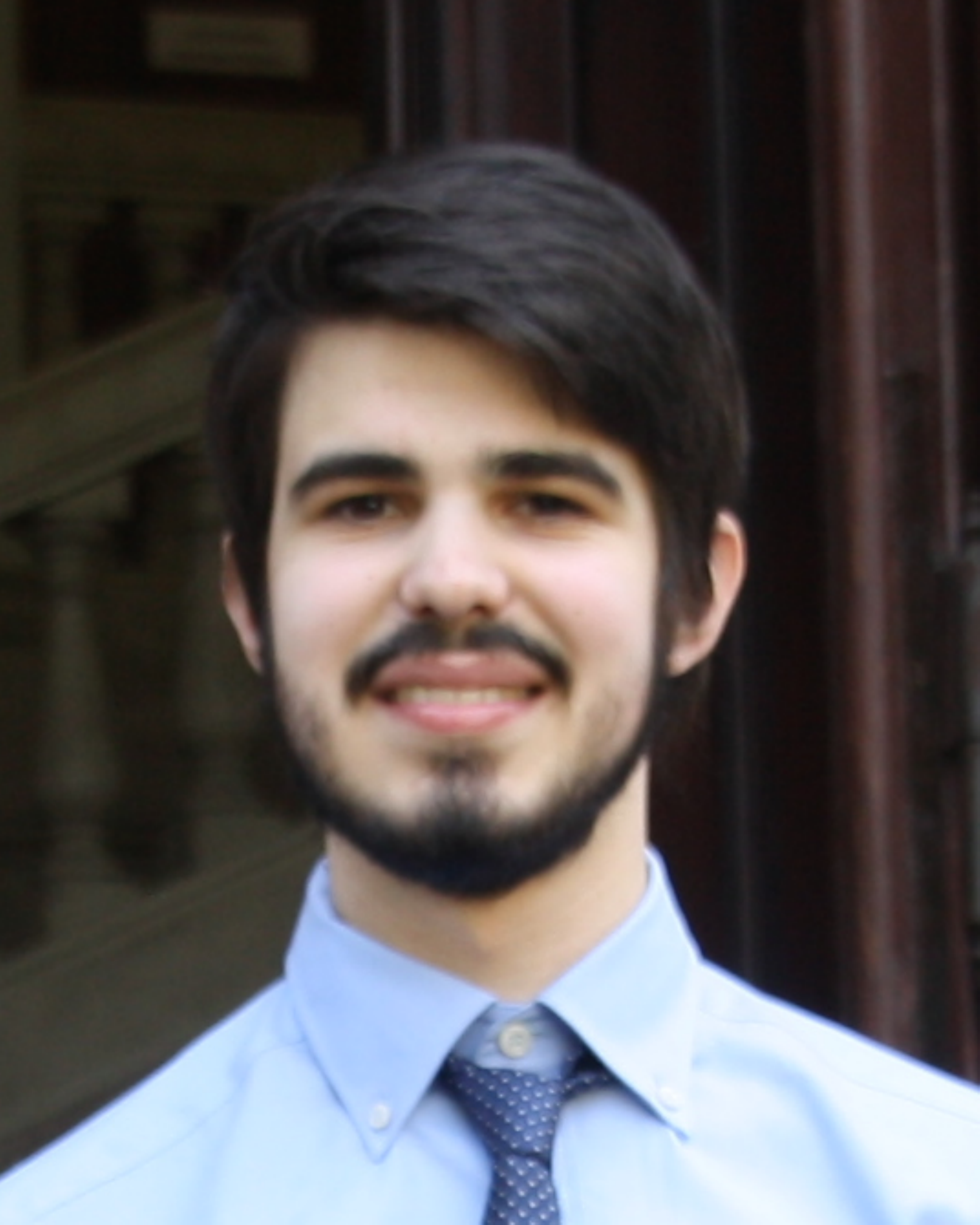}}]%
  {Giacomo Baggio} received the Ph.D. degree in Control Systems
  Engineering from the University of Padova in 2018. He is currently a
  PostDoctoral Scholar in the Department of Mechanical Engineering at
  the University of California at Riverside. From October 2015 to June
  2016, he was a Visiting Scholar in the Department of Engineering at
  the University of Cambridge. He was the recipient of the Best
  Student Paper Award at the 2018 European Control Conference. His
  current research interests lie in the area of analysis and control
  of dynamical~networks.
\end{IEEEbiography}

\begin{IEEEbiography}%
[{\includegraphics[width=1in,height=1.25in,clip,keepaspectratio]{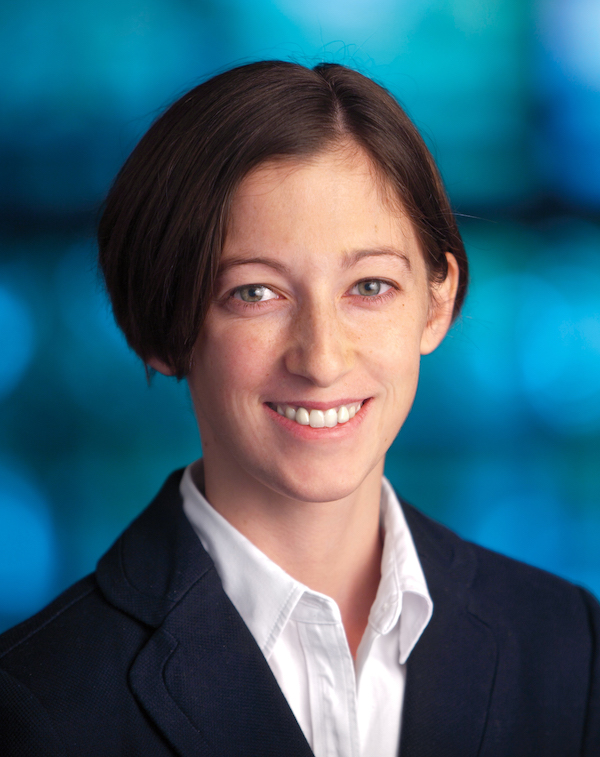}}]%
{Danielle S. Bassett}  is the Eduardo D. Glandt Faculty Fellow and Associate
  Professor in the Department of Bioengineering at the University of
  Pennsylvania. 
  She is most well known for her work blending neural
  and systems engineering to identify fundamental mechanisms of
  cognition and disease in human brain networks.
  She received a
  B.S. in physics from Penn State University and a Ph.D. in physics
  from the University of Cambridge, UK as a Churchill Scholar, and as
  an NIH Health Sciences Scholar. Following a postdoctoral position at
  UC Santa Barbara, she was a Junior Research Fellow at the Sage
  Center for the Study of the Mind. She has received multiple
  prestigious awards, including 
  American Psychological Association's
  'Rising Star' (2012), 
  Alfred P Sloan Research Fellow (2014),
  MacArthur Fellow Genius Grant (2014), Early Academic Achievement
  Award from the IEEE Engineering in Medicine and Biology Society
  (2015), Harvard Higher Education Leader (2015), Office of Naval
  Research Young Investigator (2015), National Science Foundation
  CAREER (2016), Popular Science Brilliant 10 (2016), Lagrange Prize
  in Complex Systems Science (2017), Erd\"os-R\'enyi  Prize in Network
  Science~(2018).
  She is the author of more than 200 peer-reviewed
  publications, which have garnered over 15900 citations, as well as
  numerous book chapters and teaching materials.
   \end{IEEEbiography}

\begin{IEEEbiography}%
  [{\includegraphics[width=1in,height=1.25in,keepaspectratio]{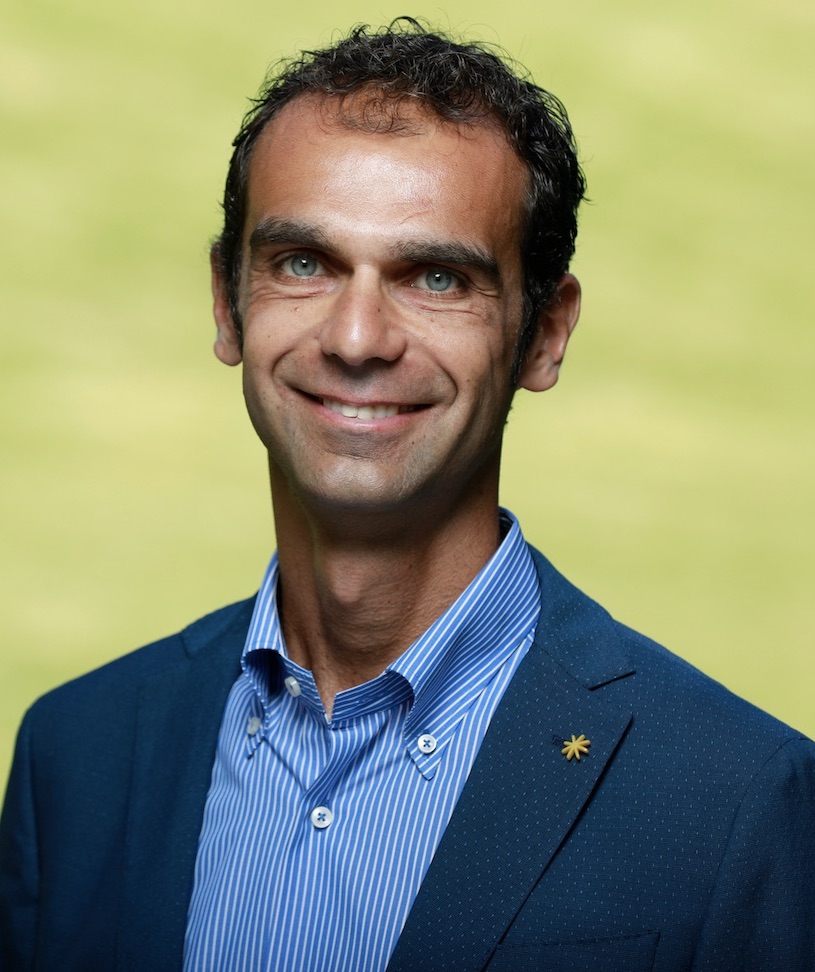}}]%
  {Fabio Pasqualetti} is an Assistant Professor in the Department of
  Mechanical Engineering, University of California at Riverside. He
  completed a Doctor of Philosophy degree in Mechanical Engineering at
  the University of California, Santa Barbara, in 2012, a Laurea
  Magistrale degree (M.Sc. equivalent) in Automation Engineering at
  the University of Pisa, Italy, in 2007, and a Laurea degree
  (B.Sc. equivalent) in Computer Engineering at the University of
  Pisa, Italy, in 2004. He has received several awards, including a
  Young Investigator Program award from ARO in 2017, and the 2016 TCNS
  Outstanding Paper Award from IEEE CSS. His main research interests
  include the analysis and control of complex networks, security of
  cyber-physical systems, distributed control, and network
  neuroscience.
\end{IEEEbiography}

\end{document}